\documentclass[12pt]{amsart}
\usepackage{comment}
\usepackage{macros}
\usepackage{enumitem}

\newtheorem*{ack}{Acknowledgements}
\newtheorem*{theorem*}{Theorem}

\usepackage{mathtools}

\begin{document}

\title
[Formal GAGA for gerbes] 
{Formal GAGA for gerbes}
\date{1 October 2023}
\vspace{-.8cm}

\author{Andrew Kresch}
\email{andrew.kresch@math.uzh.ch} 
\address{Institut f\"ur Mathematik, Universit\"at Z\"urich, Winterthurerstrasse 190, CH-8057 Z\"urich, Switzerland}

\author{Siddharth Mathur}
\email{siddharth.mathur@mat.uc.cl} 
\address{Departamento de Matem\'atica, Pontificia Universidad Cat\'olica de Chile, Santiago, Chile}

\begin{abstract}

\noindent Fix an $I$-adically complete Noetherian ring $A$ and suppose $X$ is a proper $A$-scheme. This article concerns the relationship between the Brauer group of $X$ and that of the various $X_n$ where $X_n$ is the fiber over $A/I^{n+1}$. In particular, we answer a question of Grothendieck by showing that, in positive and mixed characteristic, there are examples of $X$ with nontrivial Brauer classes that restrict to zero on all the $X_n$. We characterize such behavior, prove this cannot happen in characteristic zero, and deduce a formal GAGA statement for Brauer classes. 

\end{abstract}
\maketitle

\vspace{-.8cm}

\section{Introduction}
\label{sec:intro}

Given a family of varieties, the infinitesimal thickenings of a fiber should govern the local geometry of the family near that fiber. This principle is supported by fundamental theorems in algebraic geometry. Indeed, certain invariants of the total space of a family over a complete local Noetherian ring can be recovered from the thickenings of the closed fiber. For instance, for a family proper over the base the formal GAGA theorem establishes this for coherent sheaves via the theorem on formal functions and Grothendieck's existence theorem. Another example is formal GAGA for Picard groups (cf. Proposition \ref{elementaryPic}): if $(A, \mathfrak{m})$ is a complete local Noetherian ring, $X$ a proper $A$-scheme, and $X_n=X \times_{\Spec A} \Spec A/\mathfrak{m}^{n+1}$, then we have an isomorphism
$$\mathrm{Pic}(X) \xrightarrow{\sim} \invlim \mathrm{Pic}(X_n).$$



Grothendieck \cite[\S 3]{Brauer3} investigated the analogous map for another invariant, the Brauer group. More generally, he studied the restriction homomorphism in \'etale cohomology
\begin{equation}
\label{eqn.phi}
\phi\colon H^2(X, \mathbb{G}_m) \to \invlim H^2(X_n, \mathbb{G}_m),
\end{equation}
posed the question of injectivity of $\phi$, and suggested it may hold for torsion classes (see \cite[Rem. 3.4 (a)]{Brauer3}). Our first result explains why this is not the case in mixed and positive characteristic. Indeed, in Section \ref{sec:p} we show that N\'eron-Severi jumps characterize the non-injectivity of $\phi$ and use this to settle Grothendieck's question in Section \ref{sec:ex}. A simple version of Theorem \ref{mainp} is:



\begin{Theorem} \label{main1} Suppose $(A,\mathfrak{m},k)$ is a complete DVR with $\mathrm{char}(k)>0$ and let us denote the generic point by $\eta$. If $X$ is proper and smooth over $A$, then the following are equivalent:
\begin{itemize}
\item[(a)] The system of abelian groups $(\mathrm{Pic}(X_n))$ is Mittag-Leffler. 
\item[(b)] The homomorphism $\phi\colon H^2(X, \mathbb{G}_m) \to \invlim H^2(X_n, \mathbb{G}_m)$ is injective.
\item[(c)] The equality $\mathrm{rk}(\mathrm{NS}(X_0))=\mathrm{rk}(\mathrm{NS}(X_{\eta}))$ holds. \end{itemize}\end{Theorem}

In characteristic zero, we demonstrate the injectivity of (\ref{eqn.phi}) unconditionally (see Remark \ref{groray}, Theorem \ref{complexinjective}). We then bound the order of an inverse system of Brauer classes, and use this to deduce formal GAGA for Brauer groups in characteristic zero (see Theorem \ref{thm:GAGAbr}).



\begin{Theorem} \label{maincomplex} Suppose $(A,\mathfrak{m},k)$ is a complete local Noetherian ring with $\mathrm{char}(k)=0$. If $X$ is a proper $A$-scheme, then $(\mathrm{Pic}(X_n))$ is a Mittag-Leffler system, $\phi$ in (\ref{eqn.phi}) is injective, and it induces:
\begin{equation}
\label{eqn.tors}
H^2(X,\mathbb{G}_m)_{\mathrm{tors}} \xrightarrow{\sim} \invlim [H^2(X_n,\mathbb{G}_m)_{\mathrm{tors}}].
\end{equation}
 \end{Theorem}
 
\begin{Example} \label{ex:notsurj} The map $\phi$ in (\ref{eqn.phi}) is not surjective. For example, the group on the left is torsion when $X$ is regular, but on the right there are non-torsion elements already when $S=\Spec \mathbb{C}[[t]]$ and $X=V\times S$, with $V$ a K3 surface over $\mathbb{C}$ (see also \cite[Ex.\ 4.13]{BhattTannaka} and \cite[Rem.\ 2.1.8]{bouthiercesnavicius}). \end{Example}



The homomorphism $\phi$ in (\ref{eqn.phi}) was also studied by Geisser-Morin \cite{GeisserMorin}, under more restrictive conditions, and Binda-Porta \cite{BindaPorta}, using derived algebraic geometry. In particular, they establish the injectivity of the intermediate homomorphism
\begin{equation}
\label{eqn.introinj}
H^2(X, \mathbb{G}_m) \hookrightarrow H^2(\mathfrak{X},\mathbb{G}_m)
\end{equation}
where $\mathfrak{X}$ is the formal scheme associated with $X_0$ in $X$. In Section \ref{sec:formalinjectivity}, we will present a short alternative proof of this statement without making recourse to derived methods and which also holds for any commutative smooth affine group scheme in place of $\mathbb{G}_m$. 

The formalism of continuous cohomology in \cite{Jannsen} yields the exact sequence:
\begin{equation}
\label{eqn.introexact}
0 \to {\invlim}^1\mathrm{Pic}(X_n) \to H^2(\mathfrak{X},\mathbb{G}_m) \to \invlim H^2(X_n, \mathbb{G}_m) \to 0.
\end{equation}
So, by \eqref{eqn.introinj} if ${\invlim}^1\mathrm{Pic}(X_n)=0$, then $\phi$ is injective. This was observed, in increasing degrees of generality, by Grothendieck, Geisser-Morin, and Binda-Porta, and it allows them to conclude that $\phi$ is injective when the Picard functor is formally smooth, e.g., if $X/S$ has $1$-dimensional fibers. Unfortunately this criterion rarely applies in higher dimensions. Indeed, even for smooth projective $X$ over $A=\mathbb{C}[[t]]$, the associated Picard scheme need not be smooth.

On the other hand, ${\invlim}^1\mathrm{Pic}(X_n)=0$ if $(\mathrm{Pic}(X_n))$ has the Mittag-Leffler property, i.e., for every $m$, the image of $\mathrm{Pic}(X_n) \to \mathrm{Pic}(X_m)$ stabilizes as $n$ grows. This too fails in general; see Section \ref{sec:ex} for examples which ultimately lead to the non-injectivity of $\phi$. This should not be surprising because the Picard functor is quite large. However, if we distinguish the algebraically trivial part of the Picard functor from its N\'eron-Severi quotient we may exploit the fact that they are large in different ways: one is Noetherian and the other is a finitely generated abelian group. 

Indeed, using the Picard scheme, Lipman's mixed-characteristic analogue, and N\'eron-Popescu desingularization, we obtain results that force the Mittag-Leffler property on $(\mathrm{Pic}^0(X_n))$, even when $\mathrm{Pic}^0(-)$ is not formally smooth (see Proposition \ref{prop:mittagapproxfinite}, Theorem \ref{algclosedML}, and Corollary \ref{cor:pic0MLchar0}). In turn, this allows us to characterize the non-injectivity of $\phi$ in terms of N\'eron-Severi jumps in mixed and positive characteristic (see Theorem \ref{mainp}).


On the other hand, we show that the finite-generation of the N\'eron-Severi part implies it is eventually constant in characteristic zero (see, e.g., Proposition \ref{NSeventually}). This fact imposes cohomological properties on the Picard scheme (see Lemma \ref{NSML}) which imply three important results in characteristic zero: $(\mathrm{Pic}(X_n))$ is Mittag-Leffler, the term on the right of $(\ref{eqn.tors})$ is torsion (see Theorem \ref{thm:GAGAbr}), and the system $(\mathrm{Pic}(X_n))$ is eventually surjective in certain cases, even over non-local rings (see Proposition \ref{prop:nonlocal}). 

\begin{ack} We would like to thank K\c{e}stutis \v{C}esnavi\v{c}ius, Aise Johan de Jong, Elden Elmanto, Jack Hall,  Ariyan Javanpeykar, David Benjamin Lim, David Rydh, Vladimir Sosnilo, and Remy van Dobben de Bruyn for helpful comments and discussions. We thank Piotr Achinger for stimulating conversations and making us aware of Example \ref{ex:nonsurjp}. We are also indebted to Ofer Gabber for suggestions that led to Example \ref{ex:H1jump} and Theorem \ref{algclosedML}. The first author was partially supported by the Swiss National Science Foundation. The second author is supported by FONDECYT Regular grant No. 1230402 and was supported by the European Research Council (ERC) under the European Union’s Horizon 2020 research and innovation programme (grant agreement No. 851146). \end{ack}

\section{Preliminaries} 
\label{sec:preliminaries}
Fix a scheme $X$ and let $(\mathrm{Ab})^{\mathbb{N}}$ (respectively $(\mathrm{Sh})^{\mathbb{N}}$) denote the category of inverse systems of abelian groups (respectively abelian sheaves for the \'etale topology on $X$).

\begin{definition} (Jannsen, \cite{Jannsen}) \label{contcoho}
We denote by $H^i_{\mathrm{cont}}(X,-)$, the derived functors of the left-exact functor 
\[F\colon (\mathrm{Sh})^{\mathbb{N}} \to \mathrm{Ab}\]
sending $(G_n) \mapsto \invlim H^0(X,G_n)$.
This is the \emph{continuous cohomology} of the system $(G_n)$.
\end{definition}

\begin{remark} \label{contleray} The functor $F$ can be written as the composition of functors  
\[\invlim(-) \circ (H^0(X,-)) \colon (\mathrm{Sh})^{\mathbb{N}} \to (\mathrm{Ab})^{\mathbb{N}} \to \mathrm{Ab}.\]
Therefore the Grothendieck spectral sequence and the vanishing of $\mathrm{R}^i\invlim(-)$ for $i>1$ (for abelian groups)
 lead to short exact sequences for each $j \geq 1$:
\begin{equation}
\label{eqn.prelimexact}
0 \to {\invlim}^1H^{j-1}(X,G_n) \to H^j_{\mathrm{cont}}(X,(G_n)) \to \invlim H^j(X, G_n) \to 0.
\end{equation} \end{remark}

 For $j=1$ and $j=2$ there are descriptions of $H^j_{\mathrm{cont}}(X,(G_n))$ involving torsors and gerbes.

\begin{definition} \label{conttorsor} A $(G_n)$-\emph{torsor} on $X$ is the data of a $G_n$-torsor $P_n \to X$
and an equivariant isomorphism $\phi_n\colon P_{n+1}\times^{G_{n+1}}G_n\to P_n$, for every $n$. An \emph{isomorphism of $(G_n)$-torsors} $(P_n,\phi_n) \to (P'_n,\psi_n)$ consists of the data of $G_n$-equivariant isomorphisms $f_n\colon P_n \to P'_n$ which are compatible with $\phi_n$ and $\psi_n$ for every $n \geq 0$. \end{definition} 

Given a morphism of abelian sheaves $G\to H$, we will use the standard notation ${-}\times^G H$ for the analogous promotion of structure group for gerbes: it yields a banded $H$-gerbe when applied to a banded $G$-gerbe. 

\begin{definition} \label{contgerbe} A \emph{banded $(G_n)$-gerbe} is the data of a banded $G_n$-gerbe $\mathcal{X}_n\to X$ and an isomorphism of banded gerbes $\phi_n\colon\mathcal{X}_{n+1}\times^{G_{n+1}}G_n\to \mathcal{X}_n$ for every $n \geq 0$. An \emph{isomorphism of banded $(G_n)$-gerbes} $(\mathcal{X}_n,\phi_n) \to (\mathcal{X}'_n,\psi_n)$ consists of the data of $G_n$-banded isomorphisms $f_n\colon \mathcal{X}_n \to \mathcal{X}_n'$ which are compatible with $\phi_n$ and $\psi_n$ for every $n \geq 0$.
\end{definition}

It is a standard fact that the set of
isomorphism classes of $G$-torsors over a scheme $X$, (resp. banded $G$-gerbes over $X$),
form an abelian group, naturally isomorphic to $H^1(X,G)$ (resp.
$H^2(X,G)$).
An analogous fact also holds for $(G_n)$-torsors and gerbes. 

\begin{proposition} \label{formalabeliancoho} Let $(G_n)$ denote an inverse system of abelian sheaves over the \'etale site of a scheme $X$.
Then the isomorphism classes of $(G_n)$-torsors on $X$ are naturally identified
with $H^1_{\mathrm{cont}}(X,(G_n))$, and the
isomorphism classes of banded $(G_n)$-gerbes on $X$ are naturally identified
with $H^2_{\mathrm{cont}}(X,(G_n))$.
\end{proposition}

It is useful to note that $(\mathrm{Ab})^{\mathbb{N}}$ has enough injectives and that the injective objects $(I_n)$ of $(\mathrm{Ab})^{\mathbb{N}}$ are
characterized as those with $I_n$ injective for all $n$, and
with every $I_{n+1}\to I_n$ a split epimorphism \cite[(1.1)]{Jannsen}.
Indeed, one can then quickly deduce the triviality of all
$(I_n)$-torsors and banded $(I_n)$-gerbes for injective objects $(I_n)$.
The proof of Proposition \ref{formalabeliancoho} is then a straightforward extension of the arguments in \cite[Rem.\ III.3.5.4, Thm.\ IV.3.4.2]{giraud}; we omit
the details.

When $X$ is proper over an $I$-adically complete Noetherian ring $A$,
and we denote by $X_n$ the fiber over $\Spec(A/I^{n+1})$ for all $n$,
we may identify the \'etale sites of all the thickenings $X_n$
with that of $X_0$. 
From a commutative group scheme $G$, smooth over $X$, we obtain
the inverse system of abelian sheaves $(G_{X_n})$.
The formal scheme associated with $X_0$ in $X$ will be denoted by
$\mathfrak{X}$.

\begin{definition} \label{formalcoho} Let $X$, $\mathfrak{X}$ and $G$ be as above, then we define the $i$th cohomology of $G$ on the formal scheme $\mathfrak{X}$ to be $H^i_{\mathrm{cont}}(X_0,(G_{X_n}))$ and we will denote it by $H^i(\mathfrak{X},G)$. A \emph{$G$-torsor on $\mathfrak{X}$} is a $(G_{X_n})$-torsor, and a \emph{banded $G$-gerbe on $\mathfrak{X}$} is a banded $(G_{X_n})$-gerbe. \end{definition}

For example, if we combine these definitions with Proposition \ref{formalabeliancoho} and take $G=\mathbb{G}_{m,X}$, we obtain:

\begin{proposition}
\label{prop.H2Gm}
The group $H^2(\mathfrak{X},\mathbb{G}_m)$ is naturally identified with the
isomorphism classes of banded $\mathbb{G}_m$-gerbes on $\mathfrak{X}$.
\end{proposition}


\begin{remark} \label{bhattscholze} Returning to the general setting, provided that the transition maps $G_{n+1}\to G_n$ are epimorphisms (as is the case, e.g., for $G_n=G_{X_n}$ as above), continuous cohomology $H^j_{\mathrm{cont}}(X,(G_n))$ may be identified with the cohomology $H^j(X_{\mathrm{pro\acute{e}t}},\invlim G_n)$ of a limit sheaf on the pro-\'etale site of Bhatt and Scholze; see \cite[\S 5.6]{BhattScholze}. \end{remark}

\section{Formal injectivity for $G$-gerbes}
\label{sec:formalinjectivity}
Let $X \to S$ be a proper morphism of schemes with $S=\Spec A$ the spectrum of an $I$-adically complete Noetherian ring $A$.
We employ the notation of Section \ref{sec:preliminaries}, e.g., $\mathfrak{X}$ for the associated formal scheme.
We give a formal injectivity result, generalizing the injectivity result of Binda-Porta mentioned in the introduction.
We avoid derived methods and instead deduce formal injectivity from the following result of Bhatt and Halpern-Leistner \cite[Thm.\ 7.4]{BhattHL}.

\begin{lemma} \label{Tannaka}
Let $A$ be an $I$-adically complete Noetherian ring, and let $X$ and $\mathcal{X}$ be algebraic stacks of finite type over $A$.
Suppose that $X$ is proper over $A$, and $\mathcal{X}$ has affine diagonal.
Then, denoting the fiber over $\Spec(A/I^{n+1})$ by $X_n$, the inclusion morphisms give rise to an equivalence of categories of morphisms of stacks over $A$:
\[ \mathrm{HOM}_A(X,\mathcal{X})\cong \invlim \mathrm{HOM}_A(X_n,\mathcal{X}). \]
\end{lemma}

Originally proved using a short and direct descent-theoretic argument (cf. \cite[Prop.\ 7.5]{BhattHL}),
Lemma \ref{Tannaka} can also be deduced easily from the coherent Tannaka duality result of Hall and Rydh \cite[Thm.\ 8.4]{HallRydhCoh}, as we will now show. 
\begin{proof} Hall and Rydh's result \cite[Thm.\ 8.4]{HallRydhCoh} yields an equivalence of categories
\begin{equation}
\label{eqn.hallrydh}
\mathrm{HOM}(X,\mathcal{X})\cong
\mathrm{Funct}_{{\otimes},{\simeq}}(\mathrm{Coh}(\mathcal{X}),\mathrm{Coh}(X)),
\end{equation}
where $X$ and $\mathcal{X}$ are Noetherian algebraic stacks, and $\mathcal{X}$ has quasi-affine diagonal;
on the right in \eqref{eqn.hallrydh} is the category of right exact tensor functors and natural isomorphisms respecting the tensor structure.
Indeed, combining \eqref{eqn.hallrydh} with the Grothendieck existence theorem
yields an equivalence of categories
\[ \mathrm{HOM}(X,\mathcal{X})\cong \mathrm{Funct}_{{\otimes},{\simeq}}(\mathrm{Coh}(\mathcal{X}),\invlim \mathrm{Coh}(X_n)). \]
The category on the right is formally identified with
$\invlim \mathrm{Funct}_{{\otimes},{\simeq}}(\mathrm{Coh}(\mathcal{X}),\mathrm{Coh}(X_n))$,
and we conclude by applying
\eqref{eqn.hallrydh} to $X_n$ and $\mathcal{X}$. \end{proof}

\begin{proposition} \label{formalinjectivity}
Let $X$ be a scheme, proper over an $I$-adically complete Noetherian ring $A$, and let $G$ be a commutative group scheme, smooth and affine over $X$.
Then the natural homomorphism
$H^2(X,G) \to H^2(\mathfrak{X},G)$ is injective.
\end{proposition}

\begin{proof} Given a class $\alpha \in H^2(X,G)$, we may represent it as a banded $G$-gerbe $f\colon \mathcal{X} \to X$, and suppose $\alpha|_{\mathfrak{X}}=0$.
This means there is a compatible system of morphisms $\sigma_n\colon X_n \to \mathcal{X}$ for each $n$.
Lemma \ref{Tannaka} yields a section $X\to \mathcal{X}$ of $f$, thus $\alpha=0$.
\end{proof}

\begin{remark} \label{pastworkremark} Proposition \ref{formalinjectivity}, in the case when $G=\mathbb{G}_m$, was shown by Geisser-Morin \cite{GeisserMorin} when $X$ is regular and $X \to S$ is flat over a complete $1$-dimensional base. Using derived algebraic geometry, Proposition \ref{formalinjectivity} was proven for $G=\mathbb{G}_m$ by Binda-Porta \cite{BindaPorta}. See the proof of \cite[Lem. 4.2]{shin2020cohomological} for another precursor to Proposition \ref{formalinjectivity}. \end{remark}


\begin{remark} \label{henselianpair} One may relax the completeness condition on $A$ in Proposition \ref{formalinjectivity}: instead, we may suppose that $(A,I)$ is merely a Henselian pair (see \cite[Tag 09XE]{stacks-project}) such that $A$ is Noetherian and $A \to \widehat{A}=\invlim A/I^m$ has regular geometric fibers (e.g. if $A$ is excellent). This is because the natural map $H^2(X,G) \to H^2(X_{\widehat{A}},G)$ is injective by \cite[Lem.\ 2.1.3]{bouthiercesnavicius}. \end{remark}

\section{The Picard functor and N\'eron-Severi group}
\label{sec:picardfunctor}
Fix a proper morphism of Noetherian schemes $f\colon X \to S$,
and for a scheme $T$ with a morphism $T\to S$
we let $f_T\colon X_T\to T$ denote the base change of $f$.
We make some brief observations about the Picard functor $\underline{\mathrm{Pic}}_{X/S}$,
\[
\underline{\mathrm{Pic}}_{X/S}(T)=H^0(T,\mathbf{R}^1(f_T)_*\mathbb{G}_m),
\]
and the Picard stack $\underline{\mathcal{P}ic}_{X/S}$ which associates to $T \to S$ the groupoid of line bundles on $X_T$.

\begin{definition} \label{numtrivial} Let $X \to S$ and $T \to S$ be as above.
\ \begin{enumerate} 
\item A line bundle $L$ on a scheme $X_0$,
proper over an algebraically closed field $k$,
is said to be \emph{numerically trivial} (resp.\ \emph{algebraically trivial})
if, for every curve $C \subset X_0$, the degree of $L|_C$ is zero (resp.\ if
$L$ and $\mathcal{O}_{X_0}$ appear as fibers of some line bundle
on $(X_0)_T$, where $T$ is a connected $k$-scheme of finite type).
\item We say a line bundle $L$ on $X$ is
\emph{numerically trivial over $S$} (resp.\ \emph{algebraically trivial over $S$})
if $L_s=L|_{X_s}$ is numerically trivial for every geometric point $s\colon \Spec k \to S$ (resp.\ if $L_s=L|_{X_s}$ is algebraically trivial for every geometric point). 

\item There is a subgroup (resp.\ subsheaf) $\mathrm{Pic}^{\tau}(X_T) \subset \mathrm{Pic}(X_T)$ (resp.\ $\underline{\mathrm{Pic}}^{\tau}_{X/S} \subset \underline{\mathrm{Pic}}_{X/S}$) which consists of those line bundles which are numerically trivial over $T$
(resp.\ with sections over $T$ \'etale locally represented by line bundles which are numerically trivial over schemes, \'etale over $T$).
Similarly, there is a subgroup (resp.\ subsheaf) $\mathrm{Pic}^{0}(X_T) \subset \mathrm{Pic}(X_T)$ (resp.\ $\underline{\mathrm{Pic}}^{0}_{X/S} \subset \underline{\mathrm{Pic}}_{X/S}$), where the condition of algebraic triviality replaces numerical triviality. \end{enumerate}
\end{definition}

\begin{remark} \label{rem:tau}
The natural inclusion $\underline{\mathrm{Pic}}^{\tau}_{X/S} \subset \underline{\mathrm{Pic}}_{X/S}$ is representable by open immersions (see \cite[Exp.\ XIII, Thm. 4.7]{SGA6}).
\end{remark}

\begin{remark} \label{thmofbase} For any proper morphism $X \to S$ over a Noetherian scheme $S$, the quotient groups $\mathrm{Pic}^{\tau}(X_s)/\mathrm{Pic}^0(X_s)$ are finite for any geometric point $s \in S$, and their orders are uniformly bounded (see \cite[Exp.\ XIII, Thm. 4.6, 5.1]{SGA6}). Thus, if $[n]\colon \underline{\mathrm{Pic}}_{X/S} \to \underline{\mathrm{Pic}}_{X/S}$ is the $n$th power map, then $\underline{\mathrm{Pic}}^{\tau}_{X/S}=[n]^{-1}(\underline{\mathrm{Pic}}^0_{X/S})$ for suitable $n>0$.
\end{remark}


\begin{remark} \label{rem:rep}
Suppose $S$ is separated, $f$ is proper and flat, then $\underline{\mathcal{P}ic}_{X/S}$ is an algebraic stack, locally of finite type over $S$ with affine diagonal
(\cite[Thm.\ 1.2]{HallRydhCoh}).
If, additionally, $f$ is \emph{cohomologically flat in dimension zero} (i.e., the formation of $f_*\mathcal{O}_X$ is compatible with arbitrary base change along $S$),
then the functor $\underline{\mathrm{Pic}}_{X/S}$ is an algebraic space over $S$ (\cite[Thm.\ 7.3]{FormalModuliI}, see also \cite[Thm.\ 2.2.6]{brochardpicard} and note that Brochard's conclusion is valid in our situation because $(f_T)_*\mathbb{G}_m$ is smooth). In this case, the natural morphism $\underline{\mathcal{P}ic}_{X/S} \to \underline{\mathrm{Pic}}_{X/S}$ is a gerbe (\cite[Prop.\ 2.3.2]{brochardpicard}). Note that $f$ is cohomologically flat in dimension zero if the geometric fibers of $f$ are reduced (see \cite[Prop. 7.8.6]{EGAIII2}).
\end{remark}

It will be useful to discuss the component group of the Picard group of a (local) family. As such, we make the following definition.

\begin{definition} \label{defNS}
If $X_0$ is a scheme proper over a field $k$ then we define the \emph{N\'eron-Severi group} of $X_0$ to be the cokernel in the exact sequence
\[0 \to \mathrm{Pic}^0(X_0) \to \mathrm{Pic}(X_0) \to \mathrm{NS}(X_0) \to 0.\]
If $f\colon X \to S=\Spec A$ is a proper morphism where $(A,\mathfrak{m},k)$ is a complete local Noetherian ring,
then we define the \emph{N\'eron-Severi group of $X$ over $S$}, denoted by $\mathrm{NS}(X/S)$, via the following commutative diagram with exact rows
\[
\begin{tikzcd}
  0 \arrow[r] & r^{-1}(\mathrm{Pic}^0(X_0)) \arrow[d] \arrow[r] & \mathrm{Pic}(X) \arrow[d,"r"] \arrow[r] & \mathrm{NS}(X/S) \arrow[d] \arrow[r] & 0  \\
  0 \arrow[r] & \mathrm{Pic}^0(X_0) \arrow[r] & \mathrm{Pic}(X_0) \arrow[r] & \mathrm{NS}(X_0) \arrow[r] & 0
\end{tikzcd}
\]
where $X_0=X \times_{\Spec A} \Spec k$.
\end{definition}

\begin{remark} \label{rem:rankmakessense}
Note that in the above definition $\mathrm{NS}(X/S) \subset \mathrm{NS}(X_0)$ and since the latter is a finitely generated abelian group (see \cite[Thm.\ 3.4.1 (i)]{brochardfiniteness}), the rank of $\mathrm{NS}(X/S)$ is a well-defined non-negative integer.
\end{remark}

\begin{remark} \label{rem:tauneron} The group $\mathrm{Pic}(X)/\mathrm{Pic}^{\tau}(X)$ is finitely generated and has the same rank as $\mathrm{NS}(X/S)$. We may see this as follows: first, note that $r^{-1}(\mathrm{Pic}^0(X_0)) \subset \mathrm{Pic}^{\tau}(X)$. Indeed, if a line bundle is numerically trivial on the closed fiber of the proper morphism $X \to \Spec A$, then it is numerically trivial on all fibers (see Remark \ref{rem:tau}). Thus, the group $\mathrm{Pic}(X)/\mathrm{Pic}^{\tau}(X)$ is a quotient of $\mathrm{NS}(X/S)$ and the kernel is a finite group. Indeed, $\mathrm{Pic}^{\tau}(X)/r^{-1}(\mathrm{Pic}^0(X_0))$ is finite because it injects into $\mathrm{Pic}^{\tau}(X_0)/\mathrm{Pic}^0(X_0)$ and the latter is a finite group by Remark \ref{thmofbase}. \end{remark}

We end this section with two results, Proposition \ref{semiabelianfiber} and Lemma \ref{lem:tauclosed}, that are essentially due to Grothendieck; see \cite[Cor. 2.3, Cor. 2.7, Rem. 2.9]{TDTEVI}.

\begin{proposition} \label{semiabelianfiber}
Fix a reduced Noetherian scheme $S$ and let $X \to S$ be a proper and flat morphism of schemes which is cohomologically flat in dimension zero. Suppose that for each geometric point $s\colon \Spec k \to S$,
the resulting Picard scheme $\underline{\mathrm{Pic}}^0_{X_s/k}$ is a semi-abelian variety.
Then $\underline{\mathrm{Pic}}^0_{X/S}$ is a smooth algebraic space over $S$.
\end{proposition}

\begin{proof}
The claim is local on $S$, so we may assume that $S=\Spec A$ with some prime
$\ell$ invertible in $A$.
For any $n$, the deformation and obstruction spaces of a line bundle on $X$ are uniquely $\ell^n$-divisible, this shows that $\underline{\mathrm{Pic}}_{X/S}[\ell^n] \to S$ is formally \'etale and hence \'etale over $S$.
The semi-abelian hypothesis guarantees that the union of torsion $\underline{\mathrm{Pic}}_{X/S}[\ell^n]$ (for all $n>0$) is schematically dense when intersected with each fiber $\underline{\mathrm{Pic}}^0_{X_s/k}$.
In turn, this implies that $\underline{\mathrm{Pic}}_{X/S} \to S$ is universally open over each point of $\underline{\mathrm{Pic}}^0_{X_s/k}$ for each geometric point of $S$.

Since the geometric fibers $\underline{\mathrm{Pic}}^0_{X_s/k}$ are integral, \cite[Lem. 4.2.8 (i)]{brochardpicard} implies that $\underline{\mathrm{Pic}}^0_{X/S}$
is an open subspace of $\underline{\mathrm{Pic}}_{X/S}$.
It follows that $\underline{\mathrm{Pic}}^0_{X/S} \to S$ is flat, by \cite[Cor.\ 15.2.3]{EGAIV3}.
Thus $\underline{\mathrm{Pic}}^0_{X/S} \to S$ is smooth, as desired.
\end{proof}

Recall that a scheme $X$ is \emph{semi-normal} if it is reduced and every finite bijective morphism of schemes $f\colon Y \to X$ which induces an isomorphism on all residue fields is an isomorphism.

\begin{remark} \label{rem:seminormal} Let $S$ be a reduced Noetherian $\mathbb{Q}$-scheme. If the fibers of a proper and flat morphism $X \to S$ are (semi-)normal then $\underline{\mathrm{Pic}}^0_{X_s/k}$ is a (semi-)abelian variety; see \cite[Part 5, Thm. 5.4, Rem. 5.8]{zbMATH02229020} and \cite[Intro.\ and Cor.\ 8]{geisser}.  \end{remark}


Over a reduced Noetherian $\mathbb{Q}$-scheme $S$, a proper and flat morphism $X \to S$ with semi-normal fibers must have $S$-smooth $\underline{\mathrm{Pic}}^0_{X/S}$ by Proposition \ref{semiabelianfiber}. However, the next example shows that this can fail when the fibers of $X/S$ are not semi-normal.

\begin{Example} \label{ex:H1jump}
There is a flat family of projective surfaces $X\to S$ over a variety $S$, cohomologically flat in dimension zero, such that
$\underline{\mathrm{Pic}}^0_{X/S}=\underline{\mathrm{Pic}}^{\tau}_{X/S}$ has varying fiber dimension over $S$.
Let $(E,\infty)$ be an elliptic curve over $\mathbb{C}$.
We take $S=E$, the constant family $V=E\times S\to S$, and the line bundle $L=\cO_V(\Delta-\infty)$, of degree $0$ on fibers. Now consider $\mathbb{P}(\cO_V\oplus L)\to V$, take $Z$ to be the zero-section in $L\subset \mathbb{P}(\cO_V\oplus L)$, and pinch $\mathbb{P}(\cO_V\oplus L)$ along the first-order thickening $Z^{(1)}$ of $Z$ to obtain $X \to V$, with rational cuspidal curves as fibers. Now $X\to S$ has $H^1(X_s,\cO_{X_s})$ of dimension $1$ for $s \neq \infty$ and $2$ for $s=\infty$.
\end{Example}


\begin{lemma} \label{lem:tauclosed} Let $\pi \colon X \to S$ be a proper and flat morphism over a Noetherian scheme $S$. If the geometric fibers of $\pi$ are connected and normal, then $\underline{\mathrm{Pic}}^{\tau}_{X/S} \to \underline{\mathrm{Pic}}_{X/S}$ is representable by closed immersions. \end{lemma}


\begin{proof} It suffices to show the open immersion $\underline{\mathrm{Pic}}^{\tau}_{X/S} \to \underline{\mathrm{Pic}}_{X/S}$ is closed under specializations, so we may assume $S$ is a DVR. Note that $\underline{\mathrm{Pic}}_{X/S}$ is separated by \cite[Tag 0DNJ]{stacks-project}. Then $\underline{\mathrm{Pic}}_{X/S}$ is a scheme by \cite[Thm. 4B]{zbMATH03448703}. Now $\underline{\mathrm{Pic}}^0_{X/S} \to S$ has proper fibers (see \cite[Part 5, Rem. 5.8]{zbMATH02229020}) so by \cite[Cor. 15.6.8]{EGAIV3} $\underline{\mathrm{Pic}}^0_{X/S}$ is closed in $\underline{\mathrm{Pic}}_{X/S}$. By Remark \ref{thmofbase}, this implies the inclusion $\underline{\mathrm{Pic}}^{\tau}_{X/S} \subset \underline{\mathrm{Pic}}_{X/S}$ is closed. \end{proof}

\vspace{-.6cm}
\section{Mittag-Leffler systems}
\label{sec:mittagleffler}
\begin{definition} An inverse system of sets $(S_i)_{i \in \mathbb{N}}$ is said to have the \emph{Mittag-Leffler property} if for every $i \in \mathbb{N}$, there is an $N_i \in \mathbb{N}$ with $\mathrm{Im}[S_n \to S_i]=\mathrm{Im}[S_{N_i} \to S_i]$ for every $n \geq N_i$.
\end{definition}

\begin{Example} \label{finitevects} The following statements are straightforward to verify. \begin{enumerate} 
\item Any inverse system of finite sets, or finite dimensional vector spaces with linear connecting maps, is Mittag-Leffler.
\item If $(B_n), (C_n)$ are two inverse systems of sets such that $B_n \to C_n$ are surjections, and $(B_n)$ is Mittag-Leffler, then so is $(C_n)$.
\item If $(A_n), (B_n)$ are two inverse systems of sets with inclusions $\iota_n\colon A_n \to B_n$ for all $n$ and connecting maps $b_{n,m}\colon B_n \to B_m$ with the property that $b_{n,m}^{-1}(A_m)=A_n$ for every $n\geq m$, and $(B_n)$ is Mittag-Leffler, then so is $(A_n)$.
\item If 
\[0 \to (A_n) \to (B_n) \to (C_n) \to 0\]
is an exact sequence of inverse systems of abelian groups and $(A_n)$
and $(C_n)$ are Mittag-Leffler, then so is $(B_n)$. 
\item If $(A_n)$ is an inverse system of abelian groups which is Mittag-Leffler, then ${\invlim}^1A_n=0$. \end{enumerate} \end{Example}

\begin{proposition}
\label{elementaryPic}
Let $X$ be proper
over an $I$-adically complete Noetherian ring $A$, with $X_n=X \times_{\Spec A} \Spec(A/I^{n+1})$ and let $\mathfrak{X}$ denote the associated formal scheme.
In the factorization
\[ \mathrm{Pic}(X)\to \mathrm{Pic}(\mathfrak{X})\to \invlim \mathrm{Pic}(X_n). \]
of the natural restriction homomorphism, both maps are isomorphisms.\end{proposition}

\begin{proof}
The left-hand map is
an isomorphism by the Grothendieck existence theorem.
The right-hand map is an isomorphism if and only if the
${\invlim}^1$ term in \eqref{eqn.prelimexact} (for $j=1$) vanishes.
By \cite[Cor. 4.1.7]{EGAIII1} the system $(H^0(X_n,\mathcal{O}_{X_n}))$ enjoys the Mittag-Leffler property,
hence the same is true for the system $(H^0(X_n,\mathbb{G}_m))$ by Example \ref{finitevects} (3), therefore the
${\invlim}^1$ term vanishes.
\end{proof}


\begin{Example} \label{surjconnectingmaps}
Any inverse system of sets whose connecting maps are surjective automatically satisfies the Mittag-Leffler property.
As such, if $\pi \colon X \to S=\Spec A$ is proper, flat, and cohomologically flat in dimension $0$, where
$A$ a reduced $I$-adically complete Noetherian ring,
and $\underline{\mathrm{Pic}}^0_{X_s/k}$ is a semi-abelian variety for every
geometric point $s\colon \Spec k \to S$, then $(\mathrm{Pic}^0(X_n))$ is a Mittag-Leffler system by Proposition \ref{semiabelianfiber}, since $\underline{\mathcal{P}ic}^{0}_{X/S}=\underline{\mathcal{P}ic}_{X/S} \times_{\underline{\mathrm{Pic}}_{X/S}} \underline{\mathrm{Pic}}^0_{X/S} \to \underline{\mathrm{Pic}}^0_{X/S}$ is smooth (see Remark \ref{rem:rep}). 
\end{Example}

Next we show that $(\mathrm{Pic}(X_n))$ is eventually surjective under certain conditions. 

\begin{proposition} \label{prop:nonlocal} Let $\pi \colon X \to S=\Spec A$ be a proper and flat morphism over a reduced $I$-adically complete Noetherian ring $A$ containing $\mathbb{Q}$. If the fibers of $\pi$ are normal, then the connecting maps in the system $(\mathrm{Pic}(X_n))$ are eventually surjective. In particular, it is a Mittag-Leffler system and $\phi\colon H^2(X, \mathbb{G}_m) \to \invlim H^2(X_n,\mathbb{G}_m)$ is injective.\end{proposition}

\begin{proof} We begin by reducing to the case when $\pi$ has geometrically integral fibers. Since $X$ is reduced by \cite[Prop. 6.8.3]{EGAIV2}, the spectrum $S'=\Spec A'$ of its ring of global sections is reduced. Moreover, $\pi$ is cohomologically flat in dimension zero (see Example \ref{rem:rep}) so $\pi_*\mathcal{O}_X$ is a locally free $\mathcal{O}_S$-module. Since the geometric fibers of $\pi$ are reduced, it follows that the morphism $S' \to S$ is finite \'etale and therefore we may replace $S$ by $S'$. 

As explained in Remark \ref{rem:seminormal} and Example \ref{surjconnectingmaps}, the system $(\mathrm{Pic}^0(X_n))$ has surjective connecting maps. Since the obstruction to lifting a line bundle lies in a $\mathbb{Q}$-vector space and $\mathrm{Pic}^{\tau}(X_n)/\mathrm{Pic}^0(X_n)$ is torsion (see Remark \ref{thmofbase}), it follows that $(\mathrm{Pic}^{\tau}(X_n))$ also has surjective connecting maps. We set $G_n=\mathrm{Pic}(X_n)/\mathrm{Pic}^{\tau}(X_n)$, and we may conclude by showing $(G_n)$ is eventually constant. First we note that the connecting maps $G_n \to G_{n-1}$ are all inclusions, because $h\colon \mathrm{Pic}(X_n) \to \mathrm{Pic}(X_{n-1})$ satisfies $h^{-1}(\mathrm{Pic}^{\tau}(X_{n-1}))=\mathrm{Pic}^{\tau}(X_n)$ by definition.

Consider the inclusion of the generic points of $T=(\Spec A/I)_{\mathrm{red}}$, $T'=\bigsqcup_{i=1}^l \Spec K_i \to T$. Now observe that if $g\colon \mathrm{Pic}(X_0) \to \mathrm{Pic}(X_{T'})$ then $g^{-1}(\mathrm{Pic}^{\tau}(X_{T'}))=\mathrm{Pic}^{\tau}(X_0)$, because the locus on $\Spec A/I$ where a line bundle is numerically trivial is open (see Remark \ref{rem:tau}), and closed by Lemma \ref{lem:tauclosed}. It follows that $G_0$ is finitely generated because it is contained in the finitely generated group $\mathrm{Pic}(X_{T'})/\mathrm{Pic}^{\tau}(X_{T'})$ (see \cite[Thm.\ 3.4.1 (i)]{brochardfiniteness}). Thus, if $G_{n} \subset G_{n-1}$ is the inclusion of a proper subgroup, then the rank must drop because the obstruction to lifting a line bundle lies in a $\mathbb{Q}$-vector space. It follows that $(G_n)$ is eventually constant, as desired. \end{proof} 

The following example shows that the normality hypothesis above is necessary.

\begin{Example} \label{ex:gainandlose}
There is a projective family of surfaces over $\mathbb{C}[[t]]$, cohomologically flat in degree $0$, such that the system $(\mathrm{Pic}(X_n))$ is not eventually surjective.
We may obtain this from the family $X\to S$ in Example \ref{ex:H1jump}, by base change to the completion of
$\mathcal{O}_{S,\infty}$.
Setting $X_n=X\times_S\Spec(\cO_S/\mathfrak{m}_{\infty}^{n+1})$, one may calculate
\[ \Pic(X_n)\cong \mathbb{Z}^2\times E(\cO_S/\mathfrak{m}_{\infty}^{n+1}) \times \mathfrak{m}_{\infty}^n/\mathfrak{m}_{\infty}^{n+1}. \]
On the last factor, the maps in the inverse system are trivial.
\end{Example}

We record the following lemma, which will help clarify how the limiting behavior of the N\'eron-Severi group compares with that of the group of line bundles modulo numerical equivalence. 

\begin{lemma} \label{prop:nsvsnum} Let $X \to \Spec A$ be a proper morphism of schemes, where $(A,\mathfrak{m})$ is a  complete local Noetherian ring. The natural map
\[{\invlim}^i \mathrm{NS}(X_n/S_n) \to {\invlim}^i \mathrm{Pic}(X_n)/\mathrm{Pic}^{\tau}(X_n)\]
is a rank-preserving surjection between finitely generated abelian groups for $i=0$ and is an isomorphism for $i=1$. \end{lemma} 

\begin{proof} As we saw in Remark \ref{rem:tauneron}, for every $n \geq 0$, there is a rank-preserving surjection 
\[\mathrm{NS}(X_n/S_n) \to \mathrm{Pic}(X_n)/\mathrm{Pic}^{\tau}(X_n)\] which implies $\mathrm{Pic}(X_n)/\mathrm{Pic}^{\tau}(X_n)$ is a finitely generated abelian group. Since the connecting maps in the system $(\mathrm{NS}(X_n/S_n))$ are all injective, the inverse limit $\invlim \mathrm{NS}(X_n/S_n)$ is a finitely generated abelian group. Indeed, it is a subgroup of the finitely generated abelian group $\mathrm{NS}(X_0/S_0)$. Moreover, the natural connecting maps between the finite abelian groups 
\[F_n=\mathrm{Ker}[\mathrm{NS}(X_n/S_n) \to \mathrm{Pic}(X_n)/\mathrm{Pic}^{\tau}(X_n)]\]
are inclusions because the same is true for $(\mathrm{NS}(X_n/S_n))$, hence $(F_n)$ is a decreasing sequence of finite abelian groups. Taking limits and noting that $\invlim^1F_n=0$, the result follows. \end{proof}


\begin{proposition} \label{nsalg}
Let $X \to S=\Spec A$ be a proper morphism, where $(A,\mathfrak{m})$ is a complete local Noetherian ring. Then $(\mathrm{Pic}^{\tau}(X_n))$ is Mittag-Leffler if and only if $(\mathrm{Pic}^0(X_n))$ is Mittag-Leffler.  \end{proposition}


\begin{proof}
For each $n$, we have an exact sequence
\[0 \to \mathrm{Pic}^0(X_n) \to \mathrm{Pic}^{\tau}(X_n) \to G_n \to 0\]
where the $G_n$ form a descending chain of finite abelian groups. Indeed, this is true for $n=0$ by the theorem of the base (see Remark \ref{thmofbase}). Moreover, if $g\colon \mathrm{Pic}^{\tau}(X_n) \to \mathrm{Pic}^{\tau}(X_0)$ denotes the restriction map, then $g^{-1}(\mathrm{Pic}^0(X_0))=\mathrm{Pic}^0(X_n)$, by definition, and $G_n \subset G_0$.
It follows that the system $(G_n)$ is eventually constant, and we obtain
the result by (3) and (4) of Example \ref{finitevects}. \end{proof}

\section{Lifting via strong Artin approximation}
\label{sec:elkik}
In this section we show that if $X \to S=\Spec A$ is proper and flat where $A$ is a complete local Noetherian ring, then the associated inverse system of abelian groups $(\mathrm{Pic}^0(X_n))$ is Mittag-Leffler. 

 
\begin{lemma} \label{Elkik} Let $\pi\colon X \to \Spec A$ be a morphism which is locally of finite type, where $(A,\mathfrak{m},k)$ is a complete local Noetherian ring, and suppose $X$ is an algebraic stack with quasi-compact separated diagonal and affine stabilizers.

\begin{enumerate} 

\item For every $n \geq 1$, we have the equality of subsets of the set $[X(A/\mathfrak{m}^n)]$ of isomorphism classes:
\[\bigcap_{j \geq n}\mathrm{Im}([X(A/\mathfrak{m}^{j})] \to [X(A/\mathfrak{m}^n)])=\mathrm{Im}([X(A)] \to [X(A/\mathfrak{m}^n)]).\]

\item If $X$ is quasi-compact, then for every $n \geq 1$, there is an integer $\beta(n) \geq n$ such that 
\[\mathrm{Im}([X(A/\mathfrak{m}^{\beta(n)})] \to [X(A/\mathfrak{m}^n)])=\mathrm{Im}([X(A)] \to [X(A/\mathfrak{m}^n)]). \]\end{enumerate}
\end{lemma}

\begin{proof} When $X=\Spec B$ is affine, (2) follows from the strong Artin approximation property for complete local Noetherian rings \cite[Thm. 2.9]{MR1759599}. For the general case of (2), since $X$ has affine stabilizers there is a smooth cover $f\colon Y \to X$ by a finite-type $A$-scheme $Y$, which we may take to be affine, such that every field-valued point of $X$ lifts to $Y=\Spec B$ (see \cite{neeraj}). Since the map $f\colon \Spec B \to X$ is also smooth, it follows that $\Spec B(A/\mathfrak{m}^i) \to [X(A/\mathfrak{m}^i)]$ is surjective for every $i$. Therefore
\[f(\mathrm{Im}(\Spec B(A/\mathfrak{m}^j) \to \Spec B(A/\mathfrak{m}^i))=\mathrm{Im}([X(A/\mathfrak{m}^j)] \to [X(A/\mathfrak{m}^i)])\]
for every $j \geq i \geq 1$. Now (2) follows by choosing a $\beta(n)$ that works for $\Spec B$. 

For (1), suppose that $x \in \bigcap_{j \geq n}\mathrm{Im}([X(A/\mathfrak{m}^{j})] \to [X(A/\mathfrak{m}^n)])$. Since $X$ is locally of finite type, we may replace $X$ by a quasi-compact open substack which contains $x$. Thus, we may apply (2) to deduce that $x \in \mathrm{Im}([X(A)] \to [X(A/\mathfrak{m}^n)])$. Since the other inclusion is obvious, the result follows. \end{proof}


 

\begin{proposition} \label{prop:mittagapproxfinite}
Let $f\colon X \to S=\Spec A$ be a proper and flat morphism of schemes where $A$ is a complete local Noetherian ring. Then $(\mathrm{Pic}^{0}(X_n))$ is Mittag-Leffler.
\end{proposition}

\begin{proof}
The stack $\underline{\mathcal{P}ic}_{X/S}$ is an algebraic stack,
locally of finite type with affine diagonal (see Remark \ref{rem:rep}).
Moreover, the open substack $\underline{\mathcal{P}ic}^{\tau}_{X/S}$ of $\underline{\mathcal{P}ic}_{X/S}$
is of finite type over $S$ by \cite[Exp.\ XIII, Thm.\ 4.7 (iii)]{SGA6}. We conclude by using Proposition \ref{nsalg} and Lemma \ref{Elkik}(2). Indeed, apply the latter to $X= \underline{\mathcal{P}ic}^{\tau}_{X/S}$ to see that the images of $\mathrm{Pic}^{\tau}(X_m) \to \mathrm{Pic}^{\tau}(X_n)$ stabilize for $m \geq \beta(n)$ where $\beta(n)$ is as in Lemma \ref{Elkik}(2). \end{proof}



\vspace{-.5cm}
\section{A sufficient condition for non-injectivity}
\label{sec:noninjectivity}
The aim of this section is to give a sufficient condition for non-injectivity of $\phi\colon H^2(X,\mathbb{G}_m) \to \invlim H^2(X_n,\mathbb{G}_m)$ in terms of ${\ell}$-torsion in $\invlim^1 \mathrm{Pic}(X_n)$. 

\begin{proposition} \label{prop:muncoho} Let $X \to S=\Spec A$ be a proper morphism, where $A$ is an $I$-adically complete Noetherian ring, and let $\ell$ be a positive integer not divisible by the characteristic of any of the residue fields of $A/I$. Then the natural maps 
\[H^2(X,\mu_{\ell}) \to H^2(\mathfrak{X},\mu_{\ell}) \to \invlim H^2(X_n,\mu_{{\ell}})\] 
are isomorphisms, where $S_n=\Spec A/I^{n+1}$ and $X_n=X \times_S S_n$.
\end{proposition}

\begin{proof} We have the following exact sequence:
\[0 \to {\invlim}^1H^{1}(X_n,\mu_{\ell}) \to H^2(\mathfrak{X},\mu_{\ell}) \to \invlim H^2(X_n, \mu_{\ell}) \to 0.\]
However, since $\mu_{\ell}$ is an \'etale sheaf on the \'etale site common to each $X_n$,
\[H^i(X_n,\mu_{\ell}) \to H^i(X_0,\mu_{\ell})\]
is an isomorphism for each $n \geq 1$. For $i=1$ this implies ${\invlim}^1H^{1}(X_n,\mu_{\ell})=0$, so that $H^2(\mathfrak{X},\mu_{\ell}) \to \invlim H^2(X_n,\mu_{\ell})$ is an isomorphism. For $i=2$ it implies the natural homomorphism $\invlim H^2(X_n,\mu_{\ell}) \to H^2(X_0,\mu_{\ell})$ is an isomorphism. To conclude, we apply Gabber's proper base change theorem \cite[Cor. 1]{gabber} to $X/S$ to see that the composite $H^2(X,\mu_{\ell}) \to \invlim H^2(X_n,\mu_{\ell}) \cong H^2(X_0,\mu_{\ell})$ is an isomorphism, as desired. \end{proof}

\begin{corollary} \label{cor:torsionextends} Let $X \to S=\Spec A$ be a proper morphism where $A$ is an $I$-adically complete Noetherian ring and suppose $\ell$ is a positive integer not divisible by any of the characteristics of the residue fields of $A/I$.  Then the natural map is an isomorphism:
\[H^2(X,\mathbb{G}_m)[{\ell}] \xrightarrow{\sim} H^2(\mathfrak{X},\mathbb{G}_m)[{\ell}].\]
 \end{corollary}

\begin{proof} 

This follows by comparing the long exact sequences associated to the Kummer sequence on $X$ and $\mathfrak{X}$
\[
\begin{tikzcd}
  0 \arrow[r] & \mathrm{Pic}(X)/\ell\mathrm{Pic}(X) \arrow[d] \arrow[r] & H^2(X,\mu_{\ell}) \arrow[d] \arrow[r] & H^2(X,\mathbb{G}_m)[{\ell}] \arrow[d] \arrow[r] & 0  \\
  0 \arrow[r] & \mathrm{Pic}(\mathfrak{X})/\ell\mathrm{Pic}(\mathfrak{X}) \arrow[r] & H^2(\mathfrak{X},\mu_{{\ell}}) \arrow[r] & H^2(\mathfrak{X},\mathbb{G}_m)[{\ell}] \arrow[r] & 0
\end{tikzcd}
\]
We apply Propositions \ref{prop:muncoho} and \ref{elementaryPic} to see that the rightmost vertical arrow is an isomorphism. \end{proof}

\begin{corollary} \label{cor:torsioninvlim}Let $X \to S=\Spec A$ be a proper morphism where $A$ is an $I$-adically complete Noetherian ring and suppose $\ell$ is a positive integer not divisible by any of the characteristics of the residue fields of $A/I$. Then, if $\phi$ is as in (\ref{eqn.phi}), there is a natural identification 
\[({\invlim}^1\mathrm{Pic}(X_n))[\ell]=\mathrm{Ker}(\phi)[\ell].\] 
\end{corollary} 

\begin{proof} This follows by Proposition \ref{formalinjectivity}, Lemma \ref{cor:torsionextends}, and the exact sequence (\ref{eqn.introexact}).
\end{proof}




\section{Characteristic $p$ phenomena} \label{sec:p}

\begin{lemma}\label{l-divisible}
Let $f\colon X \to S=\Spec A$ be a proper morphism with $A$ a complete local Noetherian ring with residue field of characteristic $p>0$, and let $\ell$ be a positive integer not divisible by $p$.
\begin{enumerate}

\item The systems $(\mathrm{Pic}(X_n)[\ell])$ and $(\mathrm{Pic}(X_n)/\ell\mathrm{Pic}(X_n))$ are constant. 
\item There is a natural identification 
\[({\invlim}^1\mathrm{Pic}(X_n))[\ell]=\mathrm{Coker}[\mathrm{Pic}(X)/\ell\mathrm{Pic}(X) \to \mathrm{Pic}(X_0)/\ell\mathrm{Pic}(X_0)].\]
\item The abelian group ${\invlim}^1\mathrm{Pic}(X_n)$ is $\ell$-divisible.

\end{enumerate}

\noindent Moreover, the statements above hold with $\mathrm{Pic}^0(-)$ in place of $\mathrm{Pic}(-)$.
 \end{lemma}
\begin{proof} 

Note that the kernel and cokernel of the restriction map $r\colon \mathrm{Pic}(X_{n+1}) \to \mathrm{Pic}(X_n)$ are both $\mathbb{F}_p$-vector spaces. This implies any $L \in \mathrm{Pic}(X_n)[\ell]$ lifts to $X_{n+1}$ and that there is a unique such lift which is $\ell$-torsion. A similar argument shows that $(\mathrm{Pic}(X_n)/\ell\mathrm{Pic}(X_n))$ is a constant system and from this (1) follows.

 
 

For (2) and (3), we break up $\cdot \ell\colon \mathrm{Pic}(X_n) \to \mathrm{Pic}(X_n)$ into the two exact sequences:
\[0 \to \mathrm{Pic}(X_n)[\ell] \to \mathrm{Pic}(X_n) \to \ell\mathrm{Pic}(X_n) \to 0, \]
\[0 \to \ell\mathrm{Pic}(X_n) \to \mathrm{Pic}(X_n) \to \mathrm{Pic}(X_n)/\ell\mathrm{Pic}(X_n) \to 0,\]
take limits, and apply (1). The proof of these three statements for $\mathrm{Pic}^0(-)$ is identical. \end{proof}


\begin{remark} \label{rem:BPformula} Combining Lemma \ref{l-divisible} (2) and Lemma \ref{cor:torsioninvlim} yields a refinement of \cite[Thm. 1.3 (3)]{BindaPorta}. \end{remark}

\begin{Theorem}\label{mainp}
Let $f\colon X \to S=\Spec A$ be a proper morphism with $(A,\mathfrak{m})$ a complete local Noetherian ring with residue field of characteristic $p>0$, let $\ell>1$ be an integer not divisible by $p$, and consider the following statements:
\begin{itemize}
\item[(a)] The system of abelian groups $(\mathrm{Pic}(X_n))$ is Mittag-Leffler. 
\item[(b)] The homomorphism $\phi\colon H^2(X, \mathbb{G}_m) \to \invlim H^2(X_n, \mathbb{G}_m)$ is injective.
\item[(c)] The restriction of $\phi$ to $H^2(X, \mathbb{G}_m)[\ell]$ is injective.
\item[(d)] The equality $\mathrm{rk}(\mathrm{NS}(X/S))=\mathrm{rk}(\mathrm{NS}(X_0))$ holds.\end{itemize}
Then (a) $\Rightarrow$ (b) $\Rightarrow$ (c) $\Rightarrow$ (d). Moreover, if at least one of the following holds:
\begin{enumerate}
\item the residue field of $A$ is algebraically closed, or
\item the residue field of $A$ is finite, or
\item $f$ is flat.
\end{enumerate}
then (d) $\Rightarrow$ (a).
\end{Theorem}

\begin{proof} That (a) $\Rightarrow$ (b) follows from Proposition \ref{formalinjectivity}, and the statement (b) $\Rightarrow$ (c) is obvious. To show (c) $\Rightarrow$ (d) we assume $\mathrm{rk}(\mathrm{NS}(X/S))<\mathrm{rk}(\mathrm{NS}(X_0))$. Now we consider the commutative diagram for each $n \geq 0$:
\[
\begin{tikzcd}[column sep=18pt]
  0 \arrow[r] & r^{-1}(\mathrm{Pic}^0(X_0)) \arrow[d] \arrow[r] & \mathrm{Pic}(X) \arrow[d] \arrow[r] & \mathrm{NS}(X/S) \arrow[d] \arrow[r] & 0  \\
  0 \arrow[r] & \mathrm{Pic}^0(X_n) \arrow[r] & \mathrm{Pic}(X_n) \arrow[r] & \mathrm{NS}(X_n/S_n) \arrow[r] & 0 
\end{tikzcd}
\]
and by taking limits and using Proposition \ref{elementaryPic} we obtain an identification
\[\mathrm{Coker}[g\colon \mathrm{NS}(X/S) \to \invlim \mathrm{NS}(X_n/S_n)]=\mathrm{Ker}[{\invlim}^1\mathrm{Pic}^0(X_n) \to {\invlim}^1\mathrm{Pic}(X_n)].\]
Since $g$ is an injective map between finitely generated abelian groups, either it is rank-preserving or $\mathrm{Ker}[{\invlim}^1\mathrm{Pic}^0(X_n) \to {\invlim}^1\mathrm{Pic}(X_n)]$ is a finitely generated abelian group of positive rank. If the latter holds, then by the $\ell$-divisibility of ${\invlim}^1\mathrm{Pic}^0(X_n)$ (see Lemma \ref{l-divisible}), ${\invlim}^1\mathrm{Pic}(X_n)$ has nontrivial $\ell$-torsion so, in this case, we may apply Corollary \ref{cor:torsioninvlim}. Thus, we may assume that $\mathrm{NS}(X/S) \to \invlim \mathrm{NS}(X_n/S_n)$ is rank-preserving.

Now we consider the exact sequence
\begin{equation}
\label{eqn:NS}
0 \to \mathrm{NS}(X_n/S_n) \to \mathrm{NS}(X_0) \to \mathrm{NS}(X_0)/\mathrm{NS}(X_n/S_n) \to 0
\end{equation}
and note that the cokernel is a $\mathbb{Z}/p^n\mathbb{Z}$-module and therefore $\invlim \mathrm{NS}(X_0)/\mathrm{NS}(X_n/S_n)$ is a $\mathbb{Z}_p$-module. Indeed, the groups $\mathrm{NS}(X_0)/\mathrm{NS}(X_n/S_n)$ are quotients of $\mathrm{Pic}(X_0)/\mathrm{Im}[\mathrm{Pic}(X_n) \to \mathrm{Pic}(X_0)]$ and these are $\mathbb{Z}/p^n\mathbb{Z}$-modules because the obstruction to deforming a line bundle along a square-zero extension lies in an $\mathbb{F}_p$-vector space. By taking limits of the sequence \eqref{eqn:NS} above and using the fact that $\mathrm{NS}(X/S) \to \invlim \mathrm{NS}(X_n/S_n)$ is rank-preserving, our hypothesis that $\mathrm{rk}(\mathrm{NS}(X/S))<\mathrm{rk}(\mathrm{NS}(X_0))$ implies ${\invlim}^1\mathrm{NS}(X_n/S_n)$ contains a copy of $\mathbb{Z}_p/F$ where $F$ is a finitely generated abelian group of positive rank. Thus ${\invlim}^1\mathrm{NS}(X_n/S_n)$ has nontrivial $\ell$-torsion, and since ${\invlim}^1\mathrm{Pic}^0(X_n)$ is $\ell$-divisible, it follows that ${\invlim}^1\mathrm{Pic}(X_n)$ has nontrivial $\ell$-torsion, and we again conclude by Corollary \ref{cor:torsioninvlim}.

It remains to show (d) $\Rightarrow$ (a) with one of the extra hypothesis. We suppose $\mathrm{rk}(\mathrm{NS}(X/S))=\mathrm{rk}(\mathrm{NS}(X_0))$. Then for each $n \geq 0$, we have the finite-index containments
\[\mathrm{NS}(X/S) \subset \mathrm{NS}(X_n/S_n) \subset \mathrm{NS}(X_0)\]
and therefore $(\mathrm{NS}(X_n/S_n))$ is a Mittag-Leffler system. To conclude, we use Example \ref{finitevects} (4) and the fact that $(\mathrm{Pic}^0(X_n))$ is a Mittag-Leffler system. The latter follows from Proposition \ref{prop:mittagapproxfinite} in the flat case. To treat the non-flat cases, we develop further techniques in Section \ref{sec:liftingclosed} to show $(\mathrm{Pic}^0(X_n))$ is a Mittag-Leffler system (see Theorem \ref{algclosedML}). \end{proof}

\begin{proof}[Proof of Theorem \ref{main1}] The restriction map $r_{\eta}\colon \mathrm{Pic}(X) \to \mathrm{Pic}(X_{\eta})$ is an isomorphism, so there is a specialization map $\mathrm{sp}\colon \mathrm{Pic}(X_{\eta}) \to \mathrm{Pic}(X_0)$. Also 
\[\mathrm{Pic}^0(X_{\eta}) \subset \mathrm{sp}^{-1}(\mathrm{Pic}^0(X_0)) \subset \mathrm{Pic}^{\tau}(X_{\eta})\] 
where the rightmost containment follows by Remark \ref{rem:tau} and the leftmost containment follows from, for example, the proof of \cite[Prop.\ 3.3]{MaulikPoonen}. Indeed, it is shown there that the specialization map $\mathrm{sp}$ preserves algebraic equivalence. This implies $\mathrm{rk}(\mathrm{NS}(X/S))=\mathrm{rk}(\mathrm{NS}(X_{\eta}))$ and the result follows by Theorem \ref{mainp}.
\end{proof}

\begin{remark} \label{rem:question} 

\noindent Theorem \ref{mainp} answers a question implicit in the work of Binda-Porta: the family in \cite[Ex. 4.18]{BindaPorta} has ${\invlim}^1\mathrm{Pic}(X_n)=0$ because it satisfies (d). In fact, we do not know if there is a morphism $f$ with ${\invlim}^1\mathrm{Pic}(X_n) \neq 0$ and $\mathrm{Ker}(\phi)[\ell]=0$ for a single $\ell>1$ not divisible by $p$. \end{remark} 

\begin{Question} \label{q:nontors} If $f\colon X \to \Spec A$ is a proper morphism over a complete local Noetherian ring with residue characteristic $p\geq 0$, when do non-torsion (or $p$-power torsion) classes in $H^2(\mathfrak{X},\mathbb{G}_m)$ extend to $H^2(X,\mathbb{G}_m)$? \end{Question}

\vspace{-.5cm}
\section{Examples} \label{sec:ex}

In this section we record a few explicit examples where 
\[\phi\colon H^2(X,\mathbb{G}_m) \to \invlim H^2(X_n,\mathbb{G}_m)\]
is not injective when the residue field has characteristic $p>0$. 

\begin{Example} \label{ex:vL} In \cite[Section 3]{vanluijkK3}, van Luijk writes down an explicit quartic surface $X \subset \mathbb{P}^3_{\mathbb{Z}}$, flat over $\mathbb{Z}$, such that \begin{enumerate}
\item the generic fiber is a K3 surface which has geometric Picard number $1$,
\item the fibers $X_p$ over $\mathbb{F}_p$ for $p=2$, $3$ are K3 surfaces with geometric Picard number $2$, and 
\item The Picard group of $X_2$ (resp.\ $X_3$) is freely generated by a hyperplane section and a conic (resp.\ a line) defined over $\mathbb{F}_p$ (see \cite[Rem.\ 3.4]{vanluijkK3}).
\end{enumerate}
This implies, if we take $X \times \Spec \mathbb{Z}_p \to \Spec \mathbb{Z}_p$ where $p=2$, $3$, then $\phi$ is not injective. This can also be verified by directly checking that ${\invlim}^1 \mathrm{Pic}(X \times \Spec \mathbb{Z}_p/p^n\mathbb{Z}_p)\cong \mathbb{Z}_p/\mathbb{Z}$ and applying Corollary \ref{cor:torsioninvlim} to conclude. \end{Example}

\begin{Example} \label{ex:fermat} If $A=\mathbb{Z}_3[\sqrt{-1}]$, then $\phi$ is not injective when $X \subset \mathbb{P}^3_A$ is the Fermat quartic.
Indeed, Tate observed that $\mathrm{rk}(\mathrm{NS}(X_s))=22$ (see \cite[Sec.\ 3]{tatepoles}) whereas $\mathrm{rk}(\mathrm{NS}(X_{\eta}))=20$. \end{Example}

\begin{Example} \label{ex:CM} In \cite{KM}, the construction is given of the universal elliptic curve with full level $N \geq 3$ structure:
\begin{enumerate} 
\item a smooth affine curve $Y(N) \to \mathbb{Z}[1/N]$ (see \cite[Cor. 4.7.2]{KM}),
\item a (relative) elliptic curve $\mathcal{E}(N) \to Y(N)$,
\item inducing a surjective morphism to the moduli stack of elliptic curves $Y(N) \to \mathcal{M}_{1,1,\mathbb{Z}[1/N]}$.
\end{enumerate}
For explicit equations defining the above when $N=3$, see \cite[(2.2.11)]{KM}. Then for any prime $p$ not dividing $N$, there are always points $s \colon \Spec \overline{\mathbb{F}}_{p} \to Y(N)$ such that the corresponding elliptic curve is supersingular, whereas the generic such elliptic curve is not (see \cite[\S V.4]{zbMATH05549721}). Thus, there exists a morphism $S=\Spec \overline{\mathbb{F}}_p[[t]] \to Y(N)$, so that the generic fiber of $\mathcal{E}_{S}=\mathcal{E}(N) \times_{Y(N)} S \to S$ is not supersingular, but the special fiber is. So $\mathcal{E}_{S} \to S$ is an elliptic curve whose ring of endomorphisms is of rank $4$ on the special fiber and has rank $\leq 2$ generically. Then $X=\mathcal{E}_{S} \times_S \mathcal{E}_{S} \to S$ is an abelian surface over $S$ with jumping N\'eron-Severi rank. It follows that $\phi$ is not injective for this $X/S$.\end{Example}

\section{Liftability is a closed condition}
\label{sec:liftingclosed}
From Section \ref{sec:elkik} it follows that $(\mathrm{Pic}^0(X_n))$ is Mittag-Leffler when $(A,\mathfrak{m},k)$ is a complete local Noetherian ring and $X$ a proper and flat $A$-scheme. In this section, we show $(\mathrm{Pic}^0(X_n))$ is Mittag-Leffler without flatness as long as $k$ is algebraically closed or finite.



\begin{remark} \label{fieldofreps}
Recall, by Cohen's structure theorem, if $A$ is a
complete local Noetherian ring which contains a field,
then there is an isomorphism $A \cong k[[x_1,..,x_n]]/I$, where
$k$ is the residue field of $A$.
When we refer to a \emph{complete local Noetherian $k$-algebra}, we implicitly mean that $k$ is a field of representatives, i.e., that the residue field is isomorphic to $k$.
\end{remark}

In mixed characteristic, we will not be able to anchor our analysis over a field of representatives. Instead, we work over the Witt vectors, denoted by $W(B)$ for an $\mathbb{F}_p$-algebra $B$. In this situation, a construction of Lipman yields a group scheme over $k$ that resembles a Picard scheme, at least when $k$ is perfect.

\begin{Theorem} (Lipman) \label{picardlikescheme}
Let $X \to \Spec A$ be a proper morphism of schemes over a complete local Noetherian ring
$(A, \mathfrak{m},k)$,
such that $k$ is perfect.
Then for each $n$, the fpqc sheaf associated to the functor on $k$-algebras 
\[B \mapsto \mathrm{Pic}(X_n \times_{\Spec W(k)} \Spec W(B))\] 
is representable by a commutative group scheme $\mathbf{P}_{X_n}$, locally of finite type over $k$,
with functorial exact sequences for perfect fields $k'$ over $k$
\[ 0\to \mathrm{Pic}(X_n \times_{\Spec W(k)} \Spec W(k')) \to \mathbf{P}_{X_n}(k') \to \mathrm{Br}(H^0(X_n,\mathcal{O}_{X_n})_{\mathrm{red}}\otimes_k k'). \]
Moreover, there are affine morphisms
\[ r'_{n,m}\colon \mathbf{P}_{X_n}\to \mathbf{P}_{X_m} \]
for every $n\ge m\ge 0$.
Furthermore, there is a natural isomorphism $\underline{\mathrm{Pic}}_{X_0/k} \cong \mathbf{P}_{X_0}$, and the identity component $\mathbf{P}^0_{X_n}$ is of finite type over $k$.
\end{Theorem}

\begin{proof}
All except the last two sentences is contained in
\cite[Thm.\ 1.2, Thm.\ 7.5]{lipmanpicard}.
The same argument in the last paragraph of \cite[p.\ 29]{lipmanpicard} shows that $\underline{\mathrm{Pic}}_{X_0/k} \cong \mathbf{P}_{X_0}$. Moreover, by definition, there are maps $r'_{n,m}$, and they are affine because the maps $f_n\colon \mathbf{P}_{X_n} \to \mathbf{P}_{1}$ defined in \cite[p.\ 29]{lipmanpicard} are all affine (see \cite[Prop.\ 2.5]{lipmanpicard}) and $f_m \circ r'_{n,m}=f_n$. To see that the identity component of $\mathbf{P}_{X_n}$ is of finite type over $k$,
we consider
\[ r'_{n,0}\colon \mathbf{P}_{X_n}\to \mathbf{P}_{X_0}\cong \underline{\mathrm{Pic}}_{X_0/k}. \]
Now $(r'_{n,0})^{-1}(\underline{\mathrm{Pic}}^0_{X_0/k})$ is of finite type
and contains $\mathbf{P}_{X_n}^0$.
\end{proof}

\begin{Theorem} \label{algclosedML} Let $X \to \Spec A$ be a proper morphism of schemes where $A$ is a complete local Noetherian ring whose residue field $k$ is of characteristic $p>0$. If $k$ is finite or algebraically closed, then the system of abelian groups $(\mathrm{Pic}^0(X_n))$ is Mittag-Leffler. \end{Theorem}

\begin{proof} We explain the proofs of the equal and mixed characteristic cases in parallel. For $n \geq m$ the scheme-theoretic images of the restriction morphisms $r_{n,m}\colon \underline{\mathrm{Pic}}_{X_n/k} \to \underline{\mathrm{Pic}}_{X_m/k}$ (resp.\ $r'_{n,m}\colon \mathbf{P}_{X_n} \to \mathbf{P}_{X_m}$, in the mixed characteristic case)
yield a decreasing sequence of closed subschemes of
$\underline{\mathrm{Pic}}_{X_m/k}$ (resp.\ $\mathbf{P}_{X_m}$), which
we intersect with $r_{m,0}^{-1}(\underline{\mathrm{Pic}}^0_{X_0/k})$ (resp.\ $(r'_{m,0})^{-1}(\underline{\mathrm{Pic}}^0_{X_0/k})$) to obtain $Z_{n,m}$ (resp.\ $A_{n,m}$). Since $r_{m,0}^{-1}(\underline{\mathrm{Pic}}^0_{X_0/k})$ (resp. $(r'_{m,0})^{-1}(\underline{\mathrm{Pic}}^0_{X_0/k})$) is of finite type over $k$, when $k$ is finite, the associated sets of $k$-points for each $m$ are each finite and hence the system $(\mathrm{Pic}^0(X_n))$ is Mittag-Leffler. Thus, we may assume $k$ is algebraically closed. 
We claim that a line bundle $L_m$ on $X_m$ lies in $Z_{n,m}(k)$ (resp. $A_{n,m}(k)$) only if there is a line bundle $L_n \in \mathrm{Pic}^0(X_n)$ which extends $L_m$. In the equicharacteristic case, this will imply
\[Z_{n,m}(k)=\mathrm{Im}[\mathrm{Pic}^0(X_n) \to \mathrm{Pic}^0(X_m)]\]
for all $n \geq m$. In the mixed characteristic case this implies
\[A_{n,m}(k)=\mathrm{Im}[\mathrm{Pic}^0(X_n) \to \mathrm{Pic}^0(X_m)]\]
for all $n \geq m$.

Choose $L_m \in Z_{n,m}(k)$ (or, in the mixed characteristic case, $L_m \in A_{n,m}(k)$). Since each $r_{n,m}$ (resp. $r_{n,m}'$) has a closed image set-theoretically (see \cite[Exp.\ $\mathrm{VI}_{\mathrm{B}}$, Prop. 1.2]{SGA3-1}) there is a $k$-point of $\underline{\mathrm{Pic}}_{X_n/k}$ (resp. $\mathbf{P}_{X_n}$) restricting to $L_m \in \underline{\mathrm{Pic}}_{X_m/k}(k)$ (resp. restricting to $L_m \in \mathbf{P}_{X_m}(k)$), because $r_{n,m}$ (resp. $r'_{n,m}$) is of finite type and $k$ is algebraically closed. 

In the equicharacteristic case, $\mathrm{Pic}(X_i) \to \underline{\mathrm{Pic}}_{X_i/k}(k)$ is a bijection for every $i \geq 0$ since $k$ is algebraically closed. Hence if $L_m \in \underline{\mathrm{Pic}}_{X_m/k}(k)$ lies in $Z_{n,m}(k)$, then there is an $L_n \in \mathrm{Pic}^0(X_n)$ which extends $L_m$. In the mixed characteristic case, a similar argument using the exact sequence in the statement of Theorem \ref{picardlikescheme} shows we may identify $\mathbf{P}_{X_n}(k)$ with $\mathrm{Pic}(X_n)$. This implies that if $L_m \in \mathbf{P}_{X_m}(k)$ lies in $A_{n,m}(k)$, then there is an $L_m \in \mathrm{Pic}^0(X_n)$ which extends $L_m$.


By the Noetherian property the decreasing sequence of closed subschemes $Z_{n,m}$ (and $A_{n,m}$) must stabilize for large $n$ and therefore $\mathrm{Im}[\mathrm{Pic}^0(X_n) \to \mathrm{Pic}^0(X_m)]$ stabilizes, as desired. \end{proof}

We end the section with a useful observation.

\begin{lemma} \label{BrauerCart} Let $X \to S=\Spec A$ be a proper morphism of schemes where $A$ is a complete local Noetherian $k$-algebra. Then there is an $N>0$ such that for any $n \geq m \geq N$, the following diagram
\[
\begin{tikzcd}
   \mathrm{Pic}(X_n) \arrow[d] \arrow[r] & \underline{\mathrm{Pic}}_{X_n/k}(k) \arrow[d]   \\
 \mathrm{Pic}(X_m ) \arrow[r] & \underline{\mathrm{Pic}}_{X_m/k}(k)
\end{tikzcd}
\]
is cartesian. In particular, $(\mathrm{Pic}(X_n))$ is Mittag-Leffler
if $(\underline{\mathrm{Pic}}_{X_n/k}(k))$ is Mittag-Leffler.
\end{lemma}

\begin{proof} Suppose $X_n \to Y_n \to \Spec A/\mathfrak{m}^{n+1}$ is the Stein factorization. Then the $Y_n$'s are Artinian semi-local schemes which all have the same number of points, and one may check that there is an integer $N>0$ such that for every $n \geq m \geq N$, the natural morphisms $Y_m \to Y_n$ induce isomorphisms $(Y_m)_{\mathrm{red}} \to (Y_n)_{\mathrm{red}}$. Thus, the claim follows by comparing the low degree terms for the Leray spectral sequence applied to $\mathbb{G}_{m,X_i}$ for the morphism $X_i \to \Spec k$ (for $i=m,n$) and this yields the commutative diagram with exact rows:
\begin{center}
\begin{tikzcd}[column sep=18pt]
  0 \arrow[r] & \mathrm{Pic}(X_n) \arrow[d] \arrow[r] & \underline{\mathrm{Pic}}_{X_n/k}(k) \arrow[d] \arrow[r] & H^2(Y_n,\mathbb{G}_{m}) \cong H^2((Y_n)_{\mathrm{red}},\mathbb{G}_{m}) \arrow[d]  \\
 0 \arrow[r] & \mathrm{Pic}(X_m) \arrow[r] & \underline{\mathrm{Pic}}_{X_m/k}(k)  \arrow[r] & H^2(Y_m,\mathbb{G}_{m})\cong H^2((Y_m))_{\mathrm{red}},\mathbb{G}_m)
\end{tikzcd}
\end{center}

\noindent and we conclude by applying Example \ref{finitevects} (3). \end{proof}

\section{Injectivity results in characteristic zero}
\label{sec:charzero}
The aim of this section is to show the system $(\mathrm{Pic}(X_n))$ is \emph{always} Mittag-Leffler when the local ring $A$ has residue field of characteristic zero. We begin with a definition which is a geometric avatar of Definition \ref{defNS}.

\begin{definition} \label{NSgroupscheme} Fix a complete local Noetherian $k$-algebra $(A,\mathfrak{m})$ (recall Remark \ref{fieldofreps}) and let $X$ be a scheme proper over $S=\Spec A$. We define the \emph{N\'eron-Severi group scheme} over $k$ of $X_n=X\times_S \Spec A/\mathfrak{m}^{n+1}$ as follows. For $n=0$, it is the cokernel in the exact sequence of commutative group schemes
\[0 \to \underline{\mathrm{Pic}}^0_{X_0/k} \to \underline{\mathrm{Pic}}_{X_0/k} \to \underline{\mathrm{NS}}_{X_0} \to 0,\]
and for $n\geq 1$ it is the upper cokernel in the following commutative diagram with exact rows
\[\begin{tikzcd}
  0 \arrow[r] & r_n^{-1}(\underline{\mathrm{Pic}}^0_{X_0/k}) \arrow[d] \arrow[r] & \underline{\mathrm{Pic}}_{X_n/k} \arrow[d,"r_n"] \arrow[r] & \underline{\mathrm{NS}}_{X_n} \arrow[d] \arrow[r] & 0  \\
  0 \arrow[r] &  \underline{\mathrm{Pic}}^0_{X_0/k} \arrow[r] &\underline{\mathrm{Pic}}_{X_0/k} \arrow[r] & \underline{\mathrm{NS}}_{X_0} \arrow[r] & 0
\end{tikzcd}
\]
\end{definition}

\begin{remark} \label{NSschemevgroup} Note that because $r_n^{-1}(\underline{\mathrm{Pic}}^0_{X_0/k})$ contains $\underline{\mathrm{Pic}}^0_{X_n/k}$ for every $n \geq 0$, $\underline{\mathrm{NS}}_{X_n}$ is an \'etale group scheme over $k$. Also, if $k$ is algebraically closed, then there is a natural isomorphism $\underline{\mathrm{NS}}_{X_n}(k)=\mathrm{NS}(X_n/S_n)$ where $S_n=\Spec A/\mathfrak{m}^{n+1}$ (compare with Definition \ref{defNS}). \end{remark}

\begin{lemma} \label{NSpullback} Let $X \to \Spec A$ and $X_n \to S_n$ be as above, and assume that $k$ is a field of characteristic zero and let $\bar{k}$ denote an algebraic closure. Then 
\begin{enumerate} \item $r_n^{-1}(\underline{\mathrm{Pic}}^0_{X_0/k})(\bar{k})$ is a divisible group, and 
\item $r_n^{-1}(\underline{\mathrm{Pic}}^0_{X_0/k})=\underline{\mathrm{Pic}}^0_{X_n/k}$ as subgroup schemes of $\underline{\mathrm{Pic}}_{X_n/k}$. \end{enumerate} \end{lemma}

\begin{proof} We know that $\underline{\mathrm{Pic}}^0_{X_n/k} \subset r_n^{-1}(\underline{\mathrm{Pic}}^0_{X_0/k})$, and that $r_n^{-1}(\underline{\mathrm{Pic}}^0_{X_0/k})$ is of finite type over $k$ (because the morphisms $\underline{\mathrm{Pic}}_{X_n/k} \to \underline{\mathrm{Pic}}_{X_0/k}$ and $\underline{\mathrm{Pic}}^0_{X_0/k} \to \Spec k$ are of finite type, see \cite[Exp.\ XIII, Thm. 3.5, 4.7 (iii)]{SGA6}). Let $H$ be the cokernel:
\[ 0 \to \underline{\mathrm{Pic}}^0_{X_n/k} \to r_n^{-1}(\underline{\mathrm{Pic}}^0_{X_0/k}) \to H \to 0.\]
Since $H$ is \'etale and of finite type over $k$, it must be finite and hence $H(\bar{k})$ must be a finite group. Thus, if $r_n^{-1}(\underline{\mathrm{Pic}}^0_{X_0/k})(\bar{k})$ is divisible, so is its quotient $H(\bar{k})$ and therefore $H$ must be the trivial group scheme, as desired. As such, it remains to show that $r_n^{-1}(\underline{\mathrm{Pic}}^0_{X_0/k})(\bar{k})$ is divisible. 

Let $Z_n \subset \underline{\mathrm{Pic}}^0_{X_0/k}$ denote the scheme-theoretic image of $r_n\colon r_n^{-1}(\underline{\mathrm{Pic}}^0_{X_0/k}) \to \underline{\mathrm{Pic}}^0_{X_0/k}$. Then $Z_n$ is a closed subgroup scheme of $\underline{\mathrm{Pic}}^0_{X_0/k}$ and by \cite[Exp.\ $\mathrm{VI}_{\mathrm{B}}$, Prop. 1.2]{SGA3-1}, $r_n$ maps surjectively onto $Z_n$. Thus, if $I \subset \mathcal{O}_{X_n \otimes_k \bar{k}}$ is the ideal which cuts out $X_0 \otimes_k \bar{k}$, there an exact sequence of groups
\[H^1(X_n \otimes_k \bar{k},1+I) \to r_n^{-1}(\underline{\mathrm{Pic}}^0_{X_0/k})(\bar{k}) \to Z_n(\bar{k}) \to 0\]
where $1+I$ is the kernel of the surjection $\mathbb{G}_{m,X_n \times_k \bar{k}} \to \mathbb{G}_{m,X_0 \times_k \bar{k}}$. Since $1+I$ is isomorphic to a coherent sheaf (via the logarithm), $H^1(X_n \otimes_k \bar{k},1+I)$ admits the structure of a $\mathbb{Q}$-vector space so it, and its image, are both divisible. Moreover, $Z_n(\bar{k}) \subset \underline{\mathrm{Pic}}^0_{X_0/k}(\bar{k})=\mathrm{Pic}^0(X_0 \otimes_k \bar{k})$ consists of the subgroup of line bundles $L \in \mathrm{Pic}^0(X_0 \otimes_k \bar{k})$ which extend to $X_n \otimes_k \bar{k}$, i.e., those $L \in \mathrm{Pic}^0(X_0 \otimes_k \bar{k})$ with $0=o(L) \in H^2(X_n \otimes_k \bar{k},1+I)$. This is a divisible subgroup if $\mathrm{Pic}^0(X_0 \otimes_k \bar{k})$ is. Indeed, if $n>0$ is an integer and $M^{\otimes n}$ extends for $M \in \mathrm{Pic}^0(X_0 \otimes_k \bar{k})$, then so does $M$, because $o(M^{\otimes n})=no(M)=0$ if and only if $o(M)=0$ (since $H^2(X_n,1+I)$ has the structure of a $\mathbb{Q}$-vector space, as above). So it remains to show $\mathrm{Pic}^0(X_0 \otimes_k \bar{k})$ is divisible.

Note that $\underline{\mathrm{Pic}}^0_{X_0/k}$ is a connected commutative group scheme over a field of characteristic zero and is therefore also smooth over $k$. It follows that the multiplication by $m$ map is surjective for $m>0$. Indeed, by Chevalley's structure theorem, $\underline{\mathrm{Pic}}^0_{X_0/k}$ is an extension of an abelian variety $A$ by a smooth connected commutative affine group scheme $G$. By \cite[Cor. 16.15]{zbMATH06713849}, there is an isomorphism $G=T \times \mathbb{G}_a^{\oplus j}$ where $T$ is a torus. Since the multiplication by $m$ map is surjective on $A$ and on $G$, it follows that it is so for $\underline{\mathrm{Pic}}^0_{X_n/k}$ as well. So the epimorphism of sheaves $(-)^{\otimes m}\colon \underline{\mathrm{Pic}}^0_{X_0/k} \to \underline{\mathrm{Pic}}^0_{X_0/k}$ ($m>0$) gives rise to a surjective map of sections over $\bar{k}$. Therefore the abelian group $\mathrm{Pic}^0(X_0 \otimes_k \bar{k})$ is divisible. \end{proof}

\begin{lemma} \label{completedtensor} Let $X \to S=\Spec A$ be a proper morphism over a complete local Noetherian $k$-algebra $(A,\mathfrak{m})$ and suppose $k'/k$ is a field extension. Then if $S_n=\Spec A/\mathfrak{m}^{n+1}$, $X_n=X \times_S S_n$ and we set $(X_n)_{k'}=X_n \times_k k'$, there is a proper morphism $X' \to S'=\Spec A'$ over a complete local Noetherian $k'$-algebra $(A',\mathfrak{m}')$ such that the formal system of morphisms $((X_n)_{k'} \to S_n \times_k k')$ is isomorphic to $(X' \times_{S'} S_n' \to S_n')$, where $S_n'=\Spec A'/(\mathfrak{m}')^{n+1}$. \end{lemma}

\begin{proof}
Set $A'=A\,\widehat{\otimes}_k\,k'$ (completed tensor product).
This is a complete local Noetherian $k'$-algebra $(A',\mathfrak{m}')$, and after base change along $\Spec A' \to \Spec A$ we obtain a proper morphism $X' \to \Spec A'$ with the desired properties.
\end{proof}

\begin{lemma} \label{torsioncharzero} Let $X \to \Spec A$ be a proper morphism over a complete local Noetherian $k$-algebra $(A,\mathfrak{m})$ where $k$ has characteristic zero. If we set $X_n=X \times_S \Spec A/\mathfrak{m}^{n+1}$, then there is an $N>0$ such that for every $n \geq m \geq N$ the restriction map $\mathrm{Pic}(X_n) \to \mathrm{Pic}(X_m)$ induces an isomorphism on torsion subgroups. \end{lemma}

\begin{proof} For each $i \geq 0$ and $\ell>0$, if $X_i \to Y_i=\Spec B_i \to \Spec k$ is the Stein factorization, the Kummer sequence yields the following exact sequence:
\[0 \to B_i^{\times}/B_i^{\times \ell} \to H^1(X_i,\mu_{\ell}) \to \mathrm{Pic}(X_i)[\ell] \to 0.\]
Since $H^1(X_m,\mu_{\ell}) \to H^1(X_0,\mu_{\ell})$ is an isomorphism for every $\ell,m$ by topological invariance of the \'etale site, it suffices to show that for all large $n \geq m$, the maps $B_n^{\times}/B_n^{\times \ell} \to B_m^{\times}/B_m^{\times \ell}$ are surjective for every $\ell>0$. As in the proof of Lemma \ref{BrauerCart} there is an $N>0$ such that the map $B_n \to B_m$ induces an isomorphism on all residue fields for every $n \geq m \geq N$, which we denote by $k_1,..,k_c$. Then we can write 
\vspace{-.3cm}
\[0 \to 1+J_n \to B_n^{\times} \to \prod_{i=1}^c k_i^{\times} \to 0\]
where $J_n$ is the radical of $B_n$. Because the $\ell$th power map induces an isomorphism on $1+J_n$ (since it is a $\mathbb{Q}$-vector space, as above), the snake lemma implies $B_n^{\times}/B_n^{\times \ell} \to \prod (k_i/k_i^{\times \ell})$ is an isomorphism. Thus, for all such $n \geq m \geq N$, we may conclude that $B_n^{\times}/B_n^{\times \ell} \to B_m^{\times}/B_m^{\times \ell}$ is an isomorphism, as desired. \end{proof}

\begin{lemma} \label{semiabelian} Let $X \to \Spec A$ be a proper morphism over a complete local Noetherian $k$-algebra $(A,\mathfrak{m})$ where $k$ has characteristic zero. Then for each $n \geq 0$, the connected component of the Picard scheme of $X_n$ admits a semi-abelian decomposition, i.e., there is an exact sequence
\begin{equation} \label{eqn:conndecomp}
0 \to \mathbb{G}_a^{\oplus \ell_n} \to \underline{\mathrm{Pic}}^0_{X_n/k} \to E_n \to 0
\end{equation}
where $E_n$ is a semi-abelian variety. Moreover the morphisms $\underline{\mathrm{Pic}}^0_{X_n/k} \to \underline{\mathrm{Pic}}^0_{X_m/k}$ respect this decomposition and, for sufficiently large $n \geq m$, the induced maps $E_n \to E_m$ are isomorphisms. \end{lemma}

\begin{proof} The smooth, connected, and commutative groups $\underline{\mathrm{Pic}}^0_{X_n/k}$ can be written as 
\[0 \to \mathbb{G}_a^{\oplus \ell_n} \times T_n \to \underline{\mathrm{Pic}}^0_{X_n/k} \to A_n \to 0\]
where $T_n$ is a torus and $A_n$ is an abelian variety by Chevalley's structure theorem and \cite[Cor. 16.15]{zbMATH06713849}. Thus, $\underline{\mathrm{Pic}}^0_{X_n/k}$ fits into an exact sequence as in (\ref{eqn:conndecomp}). Moreover because $\mathbb{G}_a$ admits no nontrivial maps to an abelian variety or a torus, the morphisms $\underline{\mathrm{Pic}}^0_{X_n/k} \to \underline{\mathrm{Pic}}^0_{X_m/k}$ respect these decompositions. Note that this decomposition is compatible with base change and since $E_n \to E_m$ is an isomorphism if $E_n \otimes_k \bar{k} \to E_m \otimes_k \bar{k}$ is, by Lemma \ref{completedtensor} we may assume that $k$ is algebraically closed. Now since the $E_i$ are semi-abelian, it suffices to show that for large $n \geq m$ the map $E_n(k) \to E_m(k)$ induces an isomorphism on torsion subgroups. Moreover because the characteristic of $k$ is zero, the additive group has trivial torsion and thus, this is the same as showing $\underline{\mathrm{Pic}}^0_{X_n/k}(k) \to \underline{\mathrm{Pic}}^0_{X_m/k}(k)$ induces an isomorphism on torsion subgroups. However, the diagram 

\[
\begin{tikzcd}
   \underline{\mathrm{Pic}}^0_{X_n/k}(k)_{\mathrm{tors}}=[r^{-1}_n(\underline{\mathrm{Pic}}^0_{X_0/k})(k)]_{\mathrm{tors}} \arrow[d] \arrow[r] & \underline{\mathrm{Pic}}_{X_n/k}(k)_{\mathrm{tors}}=\mathrm{Pic}(X_n)_{\mathrm{tors}} \arrow[d]   \\
 \underline{\mathrm{Pic}}^0_{X_m/k}(k)_{\mathrm{tors}}=[r^{-1}_m(\underline{\mathrm{Pic}}^0_{X_0/k})(k)]_{\mathrm{tors}} \arrow[r] & \underline{\mathrm{Pic}}_{X_m/k}(k)_{\mathrm{tors}}=\mathrm{Pic}(X_m)_{\mathrm{tors}} \\
\end{tikzcd}\]
\vspace{-.8cm}

\noindent is cartesian (since the equalities on the left hold, by Lemma \ref{NSpullback}), so we may conclude by applying Lemma \ref{torsioncharzero}.\end{proof}

\begin{corollary} \label{cor:pic0MLchar0}  Let $X \to \Spec A$ be a proper morphism over a complete local Noetherian $k$-algebra $(A,\mathfrak{m})$, where $k$ has characteristic zero. Then the systems $(\underline{\mathrm{Pic}}^0_{X_n/k}(k))$ and $(\mathrm{Pic}^0(X_n))$ are both Mittag-Leffler. \end{corollary}

\begin{proof} Let $Z_{n,m} \subset \underline{\mathrm{Pic}}^0_{X_m/k}$ denote the scheme-theoretic image of 
\[r_{n,m}^0\colon \underline{\mathrm{Pic}}^0_{X_n/k} \to \underline{\mathrm{Pic}}^0_{X_m/k}.\]
Each $r_{n,m}^0$ has a closed set-theoretic image by \cite[Exp.\ $\mathrm{VI}_{\mathrm{B}}$, Prop. 1.2]{SGA3-1}. We claim that the induced map $\underline{\mathrm{Pic}}^0_{X_n/k}(k) \to Z_{n,m}(k)$ is surjective for all $n \geq m \geq N$ for some large natural number $N$, this will imply $Z_{n,m}(k)=\mathrm{Im}[\underline{\mathrm{Pic}}^0_{X_n/k}(k) \to \underline{\mathrm{Pic}}^0_{X_m/k}(k)]$. Since the descending chain of subschemes $(Z_{i,m})_{i=0}^{\infty}$ stabilizes for any given $m$ (by the Noetherian property), this would imply $(\underline{\mathrm{Pic}}^0_{X_n/k}(k))$ is Mittag-Leffler (and $(\mathrm{Pic}^0(X_n))$ as well by Lemma \ref{BrauerCart}).

Consider the semi-abelian decomposition of $\underline{\mathrm{Pic}}^0_{X_i/k}$ as in (\ref{eqn:conndecomp}) and choose an $N>0$ such that $E_n \to E_m$ is an isomorphism for all $n\geq m\geq N$ (see Lemma \ref{semiabelian}). In this case, the fiber over any $k$-point of the morphism $r^0_{n,m}\colon \underline{\mathrm{Pic}}^0_{X_n/k} \to \underline{\mathrm{Pic}}^0_{X_m/k}$ is a pseudo-torsor under $\mathrm{Ker}[\mathbb{G}_a^{\oplus \ell_n} \to \mathbb{G}_a^{\oplus \ell_m}]$. The kernel is a vector group, so since $H^1(\Spec k,\mathbb{G}_a)$ vanishes, the pseudo-torsor must have a $k$-rational point if it is non-empty.   \end{proof}

\begin{proposition} \label{NSeventually} Fix a complete local Noetherian $k$-algebra $A$, where $k$ is a field of characteristic zero, and let $X$ be a scheme proper over $\Spec A$. Then there is a $C>0$ such that for all $n \geq m\geq C$ the natural morphisms of group schemes $\underline{\mathrm{NS}}_{X_n} \to \underline{\mathrm{NS}}_{X_m}$ are isomorphisms. \end{proposition} 

\begin{proof} Since the group schemes are \'etale it suffices to show the natural morphism induces bijections $\underline{\mathrm{NS}}_{X_n}(\bar{k}) \to \underline{\mathrm{NS}}_{X_m}(\bar{k})$ for large $n \geq m$, and this is the same as showing that the natural maps 
\[\mathrm{NS}(X_n\times_k \bar{k}/S_n \times_k \bar{k}) \to \mathrm{NS}(X_m\times_k \bar{k}/S_m \times_k \bar{k})\] 
are bijections for large $n,m$. By Lemma \ref{completedtensor} applied to $\bar{k}=k'$ and replacing $A$ with $A'$ we may thus assume that $k$ is algebraically closed. Observe the following

\begin{enumerate} 
\item $(\mathrm{Pic}^0(X_n))$ is a Mittag-Leffler system of divisible abelian groups by Lemma \ref{NSpullback} and Corollary \ref{cor:pic0MLchar0}.
\item The system $(\mathrm{NS}(X_n/S_n))$ consists of finitely generated abelian groups with injective connecting maps (see Remark \ref{rem:rankmakessense}).
\item By Lemma \ref{NSpullback}, these fit into exact sequences
\[0 \to \mathrm{Pic}^0(X_n) \to \mathrm{Pic}(X_n) \to \mathrm{NS}(X_n/S_n) \to 0.\]
\end{enumerate}

For convenience, we set $G_n=\mathrm{NS}(X_n/S_n)$ and $D_n=\mathrm{Pic}^0(X_n)$. In order to prove $(G_n)$ is eventually constant, we will show that the nested sequence 
\[\dots \subset G_{2} \subset G_{1} \subset G_0\] 
can have proper inclusions at most $\mathrm{rk}(G_0)$ many times. Indeed, we will show that if an inclusion $G_{n} \subset G_{n-1}$ is proper, then $\mathrm{rk}(G_{n})<\mathrm{rk}(G_{n-1})$. Assume $G_{n} \subsetneq G_{n-1}$. Let $\alpha \in G_{n-1}$ be an element not in $G_{n}$, we will show that $\ell\alpha$ is not in $G_{n}$ for every $\ell \geq 1$. Note that this will also show that the torsion of $G_0$ lies in all the $G_n$. 

Assume for a contradiction that $\ell\alpha \in G_{n}$ for some nonzero integer $\ell$. There are line bundles $L$ on $X_{n-1}$ and $M$ on $X_{n}$ such that $[L]=\alpha$ and $[M]=\ell\alpha$. Since $\alpha$ is not in $G_{n}$ we know that $L$ does not extend to $X_{n}$. Moreover, $M|_{X_{n-1}}=L^{\otimes \ell} \otimes N$ for some line bundle $N \in D_{n-1}$. Since $D_{n-1}$ is divisible, we may write $N=(N')^{\otimes \ell}$ and therefore $M|_{X_{n-1}}=(L \otimes N')^{\otimes \ell}$. However, $L \otimes N'$ also represents $\alpha$ in $G_{n-1}$ and hence the line bundle cannot extend to $X_{n}$. In other words, if $I$ is the square-zero ideal which cuts out $X_{n-1}$ in $X_{n}$, the obstruction
\[o(L \otimes N') \in H^2(X_{n},I)\]
to extending $L \otimes N'$ to $X_{n}$ is nonzero. Since the latter is a $\mathbb{Q}$-vector space and the obstruction map $\mathrm{Pic}(X_{n-1}) \to H^2(X_{n},I)$ a group homomorphism, we must have that 
\[o(M|_{X_{n-1}})=\ell o(L \otimes N') \neq 0,\]
which contradicts the fact that $M|_{X_{n-1}}$ (and $\ell\alpha$) extends to $X_{n}$. \end{proof}

\begin{lemma} \label{NSML} Let $X \to \Spec A$ be a proper morphism over a complete local Noetherian $k$-algebra $(A,\mathfrak{m})$. Assume $k$ has characteristic zero. If we set $S_n=\Spec A/\mathfrak{m}^{n+1}$ and $X_n=X \times_S S_n$, the inverse systems of groups $(I_n)$, given by $I_n=\mathrm{Im}[\underline{\mathrm{Pic}}_{X_n/k}(k) \to \underline{\mathrm{NS}}_{X_n}(k)]$, and $(H^1(\Spec k, \underline{\mathrm{Pic}}_{X_n/k}))$ are eventually constant. \end{lemma}

\begin{proof} By Definition \ref{NSgroupscheme} and Lemma \ref{NSpullback}, for any $n \geq m \geq 0$ we obtain a commutative diagram with exact rows
\vspace{-.3cm}
\[
\begin{tikzcd}
  0 \arrow[r] & I_n \arrow[d] \arrow[r] & \underline{\mathrm{NS}}_{X_n}(k) \arrow[d] \arrow[r] & H^1(\Spec k,\underline{\mathrm{Pic}}^0_{X_n/k}) \arrow[d]  \\
  0 \arrow[r] & I_m \arrow[r] & \underline{\mathrm{NS}}_{X_m}(k)  \arrow[r] & H^1(\Spec k,\underline{\mathrm{Pic}}^0_{X_m/k})
\end{tikzcd}
\] 
By Proposition \ref{NSeventually} the middle vertical arrow is an isomorphism for all large $n$, $m$. So to show $(I_n)$ is eventually constant it suffices to show that the right-hand vertical arrow is also
an isomorphism for all large $n$, $m$. Taking cohomology of the exact sequence (\ref{eqn:conndecomp}) at $n$ and $m$ yields the diagram 
\vspace{-.3cm}
\[
\begin{tikzcd}
   H^i(\Spec k, \underline{\mathrm{Pic}}^0_{X_n/k}) \arrow[d] \arrow[r] & H^i(\Spec k, E_n) \arrow[d]   \\
 H^i(\Spec k, \underline{\mathrm{Pic}}^0_{X_m/k}) \arrow[r] & H^i(\Spec k, E_m)
\end{tikzcd}
\]
The horizontal arrows are isomorphisms for $i \geq 1$, because $H^i(\Spec k, \mathbb{G}_a)=0$, and by Lemma \ref{semiabelian} the vertical arrow on the right is an isomorphism for all large $n$, $m$. For the remaining assertion, we take cohomology of the short exact sequence
\[ 0 \to \underline{\mathrm{Pic}}^0_{X_n/k} \to \underline{\mathrm{Pic}}_{X_n/k} \to \underline{\mathrm{NS}}_{X_n} \to 0 \]
and apply Proposition \ref{NSeventually} along with the five lemma to see that the system $(H^1(\Spec k, \underline{\mathrm{Pic}}_{X_n/k}))$ is eventually constant.
\end{proof}

\begin{remark} \label{groray} In \cite[Rem. 3.4 (a)]{Brauer3}, Grothendieck proposes a strategy of Raynaud designed to produce a non-torsion class $\alpha$ for which $\phi(\alpha)=0$ ($\phi$ as in (\ref{eqn.phi})), and he also remarks that this construction does not depend on the residue characteristic. Our next result, Theorem \ref{complexinjective}, shows that any strategy which is insensitive to the characteristic must fail. \end{remark}

\begin{Theorem} \label{complexinjective} Let $f\colon X \to S=\Spec A$ be a proper morphism over a complete local Noetherian
$k$-algebra $(A,\mathfrak{m})$ where $k$ has characteristic zero. If we set $S_n=\Spec A/\mathfrak{m}^{n+1}$ and $X_n=X \times_S S_n$, then the systems of abelian groups $(\mathrm{Pic}(X_n))$, $(\underline{\mathrm{Pic}}_{X_n/k}(k))$, $(\mathrm{NS}(X_n/S_n))$ and $(\mathrm{Pic}(X_n)/\mathrm{Pic}^{\tau}(X_n))$ are Mittag-Leffler. In particular, $\phi\colon H^2(X,\mathbb{G}_m) \to \invlim H^2(X_n,\mathbb{G}_m)$ is injective. 
\end{Theorem}

\begin{proof}
For every $n \geq 0$, we have an exact sequence of abelian groups
\[0 \to \underline{\mathrm{Pic}}^0_{X_n/k}(k) \to \underline{\mathrm{Pic}}_{X_n/k}(k) \to I_n \to 0\]
where $I_n=\mathrm{Im}[\underline{\mathrm{Pic}}_{X_n/k}(k) \to \underline{\mathrm{NS}}_{X_n}(k)]$. By Corollary \ref{cor:pic0MLchar0}, the system $(\underline{\mathrm{Pic}}^0_{X_n/k}(k))$ is Mittag-Leffler. Since $(I_n)$ is Mittag-Leffler (see Lemma \ref{NSML}) it follows from Example \ref{finitevects} (4) that $(\underline{\mathrm{Pic}}_{X_n/k}(k))$ is Mittag-Leffler.
Thus by Lemma \ref{BrauerCart} we may conclude $(\mathrm{Pic}(X_n))$ is Mittag-Leffler.
By Example \ref{finitevects} (2) this implies $(\mathrm{Pic}(X_n)/\mathrm{Pic}^{\tau}(X_n))$ is
Mittag-Leffler and, by Remark \ref{rem:tauneron} and Example \ref{finitevects} (1),
that $(\mathrm{NS}(X_n/S_n))$ is Mittag-Leffler. \end{proof}

\vspace{-.3cm}
\section{Formal GAGA for Brauer classes} \label{sec:GAGAbr}

The \emph{cohomological Brauer group} of a quasi-compact scheme $X$ is $\mathrm{Br}'(X)\defeq H^2(X,\mathbb{G}_m)_{\mathrm{tors}}$. We establish a formal GAGA result for Brauer classes in characteristic zero. 

\begin{Theorem} \label{thm:GAGAbr} Let $\pi\colon X \to \Spec A$ be a proper morphism of schemes where $A$ is an $I$-adically complete Noetherian ring containing $\mathbb{Q}$. Let $X_n=X \times_{\Spec A} \Spec(A/I^{n+1})$ and suppose one of the following holds:
\begin{enumerate} \item $A$ is a complete local Noetherian $k$-algebra, or
\item $A$ is reduced and $\pi$ is flat with normal fibers. 
\end{enumerate}
Then the natural map $\psi\colon \mathrm{Br}'(X) \to \invlim \mathrm{Br}'(X_n)$
is an isomorphism.
\end{Theorem}

\begin{proof} It suffices to show that there is an $M>0$ such that if $m \geq M$ and $\alpha \in H^2(X_{m+1},\mathbb{G}_m)$ restricts to an $\ell$-torsion element in $H^2(X_m,\mathbb{G}_m)$, for any $\ell>0$, then $\alpha$ must either be $\ell$-torsion or have infinite order. This implies $\invlim \mathrm{Br}'(X_n)$ is torsion and equal to $(\invlim H^2(X_n,\mathbb{G}_m))_{\mathrm{tors}}$. The fact that $\mathrm{Br}'(X) \to \invlim \mathrm{Br}'(X_n)$ is an isomorphism now follows from  Corollary \ref{cor:torsionextends} combined with Theorem \ref{complexinjective} or Proposition \ref{prop:nonlocal}. Indeed, the latter implies $H^2(\mathfrak{X},\mathbb{G}_m) \to \invlim H^2(X_n,\mathbb{G}_m)$ is an isomorphism by (\ref{eqn.introexact}).

If (1) holds, then we may choose an $N>0$ such that the following are satisfied:
\begin{itemize} 
\item[(a)] The cokernel of $\underline{\mathrm{Pic}}_{X_{m+1}/k} \to \underline{\mathrm{Pic}}_{X_{m}/k}$ is a vector group for every $m \geq N$ (see Proposition \ref{NSeventually} and Lemma \ref{semiabelian}).
\item[(b)] The maps $H^1(\Spec k, \underline{\mathrm{Pic}}_{X_n/k}) \to H^1(\Spec k, \underline{\mathrm{Pic}}_{X_m/k})$ are isomorphisms for all $n \geq m \geq N$ (see Lemma \ref{NSML}).
\item[(c)] If $X_n \to Y_n \to \Spec A/\mathfrak{m}^{n+1}$ is the Stein factorization then $(Y_n)_{\mathrm{red}} \to (Y_m)_{\mathrm{red}}$ is an isomorphism for all $n, m \geq N$. Consequently, the maps $H^2(Y_{n},\mathbb{G}_{m}) \to H^2(Y_{m},\mathbb{G}_{m})$ are isomorphisms for all such $n,m$.
\end{itemize}
By Theorem \ref{complexinjective} we may choose an $M \geq N$ such that the image of $\underline{\mathrm{Pic}}_{X_{m}/k}(k) \to \underline{\mathrm{Pic}}_{X_{N}/k}(k)$ is stable for every $m \geq M$. Assume for $m \geq M$ there is an $\alpha \in H^2(X_{m},\mathbb{G}_m)[\ell]$ with a pre-image $\alpha'$ in $H^2(X_{m+1},\mathbb{G}_m)_{\mathrm{tors}}$ which is not killed by $\ell$. In this case, $0 \neq \beta=\ell \alpha' \in H^2(X_{m+1},\mathbb{G}_m)_{\mathrm{tors}}$ satisfies $\beta|_{X_{m}}=0$. 

Consider the short exact sequence
\begin{equation}
\label{eqn:gaga}
0 \to 1+I \to \mathbb{G}_{m,X_{m+1}} \to \mathbb{G}_{m,X_m} \to 0
\end{equation}

\noindent where $I$ is the ideal cutting out $X_m$ in $X_{m+1}$. The long exact sequence of sheaves arising from pushing forward (\ref{eqn:gaga}) along the structure morphism $f_{m+1}\colon X_{m+1} \to \Spec k$ induces a long exact sequence of stalks at the geometric point. This yields the injection:
\[\mathrm{Coker}[\mathrm{Pic}(X_{m+1} \otimes_k \bar{k}) \to \mathrm{Pic}(X_m \otimes_k \bar{k})] \to H^2(X_{m+1}\otimes_k \bar{k},1+I \otimes_k \bar{k}).\]
This is necessarily $\mathbb{Q}$-linear for $m \geq N$ because by (a) the term on the left is isomorphic to $\mathbb{G}_a^{\oplus j}(\bar{k})$ for some integer $j$. Therefore $\beta|_{X_{m+1} \otimes_k \bar{k}}$ must be zero, i.e., 
\[\beta \in \mathrm{Br}_1(X_{m+1}) \defeq \mathrm{Ker}[H^2(X_{m+1},\mathbb{G}_{m}) \to H^2(X_{m+1} \otimes_k \bar{k},\mathbb{G}_{m})].\]
By functoriality of the Leray spectral sequence, we obtain the commutative diagram
\begin{center}
\begin{tikzcd}[column sep=18pt]
    \underline{\mathrm{Pic}}_{X_{m+1}/k}(k) \arrow[d] \arrow[r, "c_{m+1}"] &  H^2(Y_{m+1},\mathbb{G}_{m}) \arrow[d, "\simeq \text{ by }\mathrm{(c)}"] \arrow[r] & \mathrm{Br}_1(X_{m+1}) \arrow[d] \arrow[r] & H^1(\Spec k, \underline{\mathrm{Pic}}_{X_{m+1}/k}) \arrow[d, "\simeq \text{ by }\mathrm{(b)}"] \\
 \underline{\mathrm{Pic}}_{X_m/k}(k)  \arrow[r, "c_m"] & H^2(Y_{m},\mathbb{G}_{m}) \arrow[r] & \mathrm{Br}_1(X_{m}) \arrow[r] & H^1(\Spec k,\underline{\mathrm{Pic}}_{X_{m}/k}).
 \end{tikzcd}
\end{center}
A diagram chase shows that if $\beta \in  \mathrm{Br}_1(X_{m+1})$ maps to zero in $ \mathrm{Br}_1(X_{m})$, then $\mathrm{Im}(c_{m+1}) \subsetneq \mathrm{Im}(c_{m})$. However, this cannot happen if $m \geq M$ because
\[\mathrm{Im}[\underline{\mathrm{Pic}}_{X_m/k}(k) \to \underline{\mathrm{Pic}}_{X_N/k}(k)]=\mathrm{Im}[\underline{\mathrm{Pic}}_{X_{m+1}/k}(k) \to \underline{\mathrm{Pic}}_{X_N/k}(k)].\]

If (2) holds then Proposition \ref{prop:nonlocal} implies there is an $M>0$ such that $\mathrm{Pic}(X_n) \to \mathrm{Pic}(X_m)$ is surjective for every $n \geq m \geq M$. Thus, by taking cohomology of the exact sequence (\ref{eqn:gaga})
we see that the kernel of $H^2(X_{m+1},\mathbb{G}_m) \to H^2(X_{m},\mathbb{G}_m)$ is a $\mathbb{Q}$-vector space. Therefore, any element $\alpha' \in H^2(X_{m+1},\mathbb{G}_m)_{\mathrm{tors}}$ which restricts to an $\ell$-torsion element in $H^2(X_m,\mathbb{G}_m)$ must be $\ell$-torsion as long as $m \geq M$. 
\end{proof}

\begin{remark} \label{rem:sameAZ} If $X$ is projective over $\Spec A$, a result of Gabber (see \cite[Thm. 4.2.1]{zbMATH07384449}) implies that the cohomological Brauer group coincides with the Brauer group, $\mathrm{Br}_{\mathrm{Az}}(X)$ of Azumaya algebras up to endomorphisms of a vector bundle. In this setting, Theorem \ref{thm:GAGAbr} can be expressed as $\mathrm{Br}_{\mathrm{Az}}(X) \xrightarrow{\sim} \invlim \mathrm{Br}_{\mathrm{Az}}(X_n)$. \end{remark}

\begin{Example} \label{ex:nonsurjp} The analog of Theorem \ref{thm:GAGAbr} fails in characteristic $p>0$, even for trivial families. Set $S=\Spec k[[t]]$ and $X=V \times S$ where $V$ is an ordinary $K3$ surface over a separably closed field $k$ of characteristic $p>0$, then there is a natural identification for each $n>0$ (see, e.g., \cite[Def. 18.3.3]{zbMATH06617332}):
\[\mathrm{Ker}[H^2(X_n, \mathbb{G}_m) \to H^2(V, \mathbb{G}_m)] \cong  1+tk[t]/(t^{n+1}) \subset (k[t]/(t^{n+1}))^\times.\]
Thus, the group $\invlim \mathrm{Br}'(X_n)$ is not torsion, so $\psi$ as in Theorem \ref{thm:GAGAbr} cannot be surjective. \end{Example}

\begin{remark} \label{rem:charpgaga} The surjectivity of $\psi$ in Theorem \ref{thm:GAGAbr} for prime-to-$p$ torsion still holds: if $X$ is a proper $A$-scheme where $A$ is a complete local Noetherian ring with residue field of characteristic $p>0$, then $\mathrm{Br}'(X)[p'] \to \invlim [ \mathrm{Br}'(X_n)[p']]$ is surjective, where $\mathrm{Br}'(X)[p']$ denotes the prime-to-$p$ torsion subgroup of $\mathrm{Br}'(X)$.  Indeed, let $\ell$ be a positive integer not divisible by $p$. By the $\ell$-divisibility of ${\invlim}^1\mathrm{Pic}(X_n)$ (see Lemma \ref{l-divisible}), any $\ell$-torsion element of $\invlim  \mathrm{Br}'(X_n)$ lifts to $H^2(\mathfrak{X},\mathbb{G}_m)[\ell]$. By Corollary \ref{cor:torsionextends} this is isomorphic to $\mathrm{Br}'(X)[\ell]$. To conclude the argument, we observe that $\invlim [ \mathrm{Br}'(X_n)[p']]$ is torsion. This holds, because for every $n$ the kernel of $H^2(X_{n+1},\mathbb{G}_m) \to H^2(X_{n},\mathbb{G}_m)$ is an $\mathbb{F}_p$-vector space.
\end{remark}

\vspace{-.5cm}

\bibliography{mybib}{}
\bibliographystyle{plain}

\end{document}